\documentclass[a4paper,11pt,leqno  ]{amsart}	% Option draft zeigt label an
\usepackage{hyperref}
\usepackage{graphicx}
\hypersetup{colorlinks, linkcolor=blue, urlcolor=red, filecolor=green, citecolor=blue}

\usepackage{times}

\usepackage{amsmath}
\usepackage{amsthm}
\usepackage{amssymb}
\usepackage{enumerate}
\usepackage{pstricks}
\usepackage{pst-node}

\theoremstyle{plain}
\newtheorem{theorem}{Theorem}[section]

\newtheorem{lemma}[theorem]{Lemma}
\newtheorem{proposition}[theorem]{Proposition}

\theoremstyle{definition}

\newtheorem{defi}{Definition}[section]

\theoremstyle{remark}
\newtheorem{remark}{Remark}

\newcommand\Rr{\mathbb{R}}
\newcommand\Zz{\mathbb{Z}}

\newcommand\ud{\,\mathrm{d}}

\newcommand\card{\operatorname{\mathrm{card}}}

\newcommand\supp{\mathop{\operatorname{supp}}}

\newcommand\diam{\operatorname{diam}}

\newcommand{\vol}{\mathrm{vol}}

\newcommand\dps{\displaystyle}
\newcommand\calA{\mathcal{A}}
\newcommand\calO{\mathcal{O}}

\newcommand\calP{\mathcal{P}}
\newcommand\calG{\mathcal{G}}

\newcommand\calN{\mathcal{N}}
\newcommand\calL{\mathcal{L}}

\newcommand\calT{\mathcal{T}}
\newcommand\calD{\mathcal{D}}

\newcommand\M{\mathbb{M}}

\newcommand\R{\mathbb{R}}
\newcommand\N{\mathbb{N}}
\newcommand\Z{\mathbb{Z}}
\newcommand{\e}{\varepsilon}
\newcommand{\ho}{\mathrm{hom}}

\newcommand{\step}[2]{\medskip \noindent \textit{Step}~#1. #2. \newline}

\newcommand\frakS{\mathfrak S}
\newcommand\Po{\mathcal{P}}
\newcommand\Set{\calA_{lf}}

\newcommand{\eq}[1]{$(\ref{#1})$}

\newcommand\Ee{\mathbb{E}}
\newcommand\calR{\mathcal{R}}
\newcommand\calC{\mathcal{C}}
\newcommand\calY{\mathcal{Y}}
\newcommand\Lip{\mathrm{Lip}}

\newcommand\RS{V}

\title[Random parking, Euclidean functionals, and rubber elasticity]{Random parking, Euclidean functionals, and rubber elasticity}
\date{\today}

\author{Antoine Gloria \and Mathew D. Penrose}
\address{Mathew D. Penrose\\Department of Mathematical Sciences\\ University of Bath \\ UK}
\email{masmdp@maths.bath.ac.uk}
\address{Antoine Gloria \\ Project-team SIMPAF \& Laboratoire Paul Painlev\'e UMR 8524\\  INRIA Lille - Nord Europe \& Universit\'e Lille 1\\ Villeneuve d'Ascq, France}
\email{antoine.gloria@inria.fr}

\begin{document}

\maketitle

\begin{abstract}
We study subadditive functions of the random parking model previously analyzed by the second author.
%in \cite{Penrose-01}.
In particular, we consider local functions $S$ of subsets of $\R^d$ and of point sets that are (almost) subadditive
in their first variable.
Denoting by $\xi$ the random parking measure in $\R^d$, and by $\xi^R$ the random parking measure in the cube $Q_R=(-R,R)^d$,
we show, under some natural assumptions on $S$, that there exists 
a constant $\overline{S}\in \R$ such that 
$$
\lim_{R\to +\infty} \frac{S(Q_R,\xi)}{|Q_R|}\,=\,\lim_{R\to +\infty}\frac{S(Q_R,\xi^R)}{|Q_R|}\,=\,\overline{S}
$$
almost surely.
If $\zeta \mapsto S(Q_R,\zeta)$ is the counting measure of $\zeta$ in $Q_R$, then 
we retrieve the result by the second author on the existence of the jamming limit. 
The present work generalizes this result to a wide class of (almost) subadditive functions.
In particular, classical Euclidean optimization problems as well as the discrete model for rubber previously studied by Alicandro, Cicalese, and the first author 
%in \cite{Alicandro-Cicalese-Gloria-07b} 
enter this class of functions.
In the case of rubber elasticity, this yields an approximation result for the continuous energy density 
associated with the discrete model at the thermodynamic limit,
as well as a generalization to stochastic networks generated on bounded sets.

\vspace{10pt}

\vspace{6pt}

\end{abstract}

%35B27 homogenization; partial differential equations in media with periodic structure
%49J45 methods involving semicontinuity and convergence; relaxation
%74Q05 homogenization, determination of effective properties

% [It would be better not to include references to the bibliography in
% \Comment{MP}
% the abstract since this may sometimes get published or
% displayed online separately from the rest of the paper.]

\tableofcontents

\section{Introduction and informal statement of the results}

%In \cite{Penrose-01}, the second author has studied 
R\'enyi's model of
random parking (also  known as 
random sequential adsorption or random sequential packing),
is defined in $d$ dimensions
%Random Sequential Adsorption (RSA) models.
%The R\'enyi model is 
as follows.
%  Unit 
A parameter $\rho_0 >0$ is specified, and
open balls 
$B_{1,R},B_{2,R},\dots$ of radius $\rho_0$, arrive sequentially and uniformly
at random
in the $d$-dimensional cube $Q_R := (-R,R)^d$, subject to non-overlap
until saturation occurs.
It is known \cite{Penrose-01} that the random parking measure $\xi^R$
in $Q_R$ (i.e. the point measure representing the locations
of the balls at saturation)
converges weakly
in the sense of measures to 
a measure $\xi$ in $\R^d$, called the random parking measure in $\R^d$, and that 
there exists $\lambda \in \R^+$ such that
\begin{equation*}
\lim_{R\to \infty}\frac{\xi(Q_R)}{|Q_R|}\,=\,\lambda
\end{equation*}
almost surely --- which yields the existence of a deterministic averaged density of packed balls in $\R^d$.
Perhaps of greater interest, however, is the existence of 
the limit of 
$
\tfrac{\xi^R(Q_R)}{|Q_R|},
$
which would be the thermodynamic limit of the averaged density of packed balls in the domains $Q_R$.
Using quantitative properties of ``stabilization'' of the random 
parking measures, it is proved in \cite{Penrose-01} that
%the second author
%has proved the following: 
in probability,
%Almost surely,
%
\begin{equation}\label{1:eq:jamming}
\lim_{R\to \infty}\frac{\xi^R(Q_R)}{|Q_R|}\,=\,\lambda.
\end{equation}
That is, not only does the measure $\xi^R$ converge weakly to $\xi$,
but also the averaged density of packed balls converges to
$\lambda$ (the jamming limit). Moreover, if
the arrivals processes for different $R$
are coupled by being all derived from a single
Poisson process in space-time,  the convergence
\eq{1:eq:jamming}
holds almost surely, as does the convergence of $\xi^R$ to $\xi$
(in fact, almost surely for all $r > 0$ we have
$\xi^R\cap {Q_r} = \xi\cap{Q_r}$ for all large enough $R$). 

\medskip
The first part of this paper is
concerned with the extension of \eqref{1:eq:jamming} to more general functions
of $\xi^R$, besides 
the total measure of $\xi^R$.
%which is considered in \eqref{1:eq:jamming}.
% the second author considers the 
%counting measure associated with
%In what follows 
We consider \emph{local subadditive} functions $S$ of bounded open sets and point sets, that is 
such that for all open bounded disjoint subsets $D,D_1,\dots,D_n$ of $\R^d$, and for every
point set $\zeta$ in some fixed class (i.~e. satisfying
the non-empty space and hard-core conditions, see Section~\ref{sec:result}), we have (see Theorem~\ref{th:sub} for milder conditions) 
\begin{eqnarray*}
S(D,\zeta)&=&S(D,\zeta_{|D}),\\
S(\cup_{i=1}^nD_i,\zeta)&\leq &\sum_{i=1}^n S(D_i,\zeta).
\end{eqnarray*}
Under the following further assumptions on $S$:
%  (see Theorem~\ref{th:sub} for milder statements):
%
\begin{itemize}
%\item non-negativity {\bf [do we need this?]} \Comment{MP} ;
\item uniform boundedness: there is a constant $C$ 
%depending only on $S$
such that $|S(D,\zeta)|\leq C |D|$ for all $D, \zeta$; 
\item insensitivity to boundary effects: there exists $0<\alpha<1$ such that
$\tfrac{S(Q_R,\zeta)-S(Q_{R-R^\alpha},\zeta)}{|Q_R|}=o(1)$;
\end{itemize}
we shall prove that there exists deterministic
$\overline{S}\in \R^+$ such that almost surely
\begin{equation}\label{1:eq:sub-limit}
\lim_{R\to +\infty}\frac{S(Q_R,\xi^R)}{|Q_R|}\,=\,\lim_{R\to +\infty}\frac{S(Q_R,\xi)}{|Q_R|}\,=\,\overline{S}.
\end{equation}
Note that \eqref{1:eq:jamming} is a particular case of \eqref{1:eq:sub-limit} for the \emph{additive} function $S(D,\zeta):=\zeta(D)=\int_D\ud \zeta(x)$.
Our proof of \eqref{1:eq:sub-limit} differs from the proof of \eqref{1:eq:jamming} in \cite{Penrose-01} in two respects.
First, the additivity of the counting measure allows one to appeal to an ergodic theorem, whereas subadditivity requires the use of
subadditive ergodic theorems (such that of Akcoglu and Krengel  \cite{Akcoglu-Krengel-81}), and the ergodicity of
the point process itself (which is used in a weaker form in \cite{Penrose-01}). Second, and more importantly, the additivity of $D\mapsto \zeta(D)$ 
and the uniform bound $0\leq \zeta(D)\leq C|D|$ imply that the contribution of any subset $D'$ of $D$ to $S(D,\zeta)$ is uniformly
bounded by $C|D'|$ --- which is crucial in \cite{Penrose-01}.
In the subadditive case, this does not hold: The contribution of a subset $D'$ of $D$ to $S(D,\zeta)$ is \emph{not}
a priori bounded by $C|D'|$. This compels us to appeal to the stabilization properties introduced by Schreiber, Yukich, and the second author
in \cite{SPY} --- which are complementary to the ones in \cite{Penrose-01}. 
The rest of the proof relies on a percolation argument and the Borel-Cantelli lemma.

\medskip
In the second part of this work we apply this general result to classical Euclidean optimization problems on the random parking measure.
In particular, we shall prove a so-called umbrella theorem, which allows to cover at once a wide class of problems, including the traveling salesman,
the minimum spanning tree, and the minimal matching. 
This is one of the first examples where subadditive arguments are combined with stabilization properties.
We refer to the recent survey of Yukich \cite{Yukich-10}, where both types of arguments are presented and used --- although not combined.

\medskip
The third part of this work is concerned with  another application of our general result: the discrete model for rubber 
introduced in \cite{Gloria-LeTallec-Vidrascu-08b}, whose thermodynamic 
limit is studied in \cite{Alicandro-Cicalese-Gloria-07b}.
In particular, as recalled in Subsection~\ref{sec:rubber}, the model under investigation is based on the notion
of stochastic lattices (in the sense of random point processes) and on the associated Delaunay tessellations.
%\Comment{MP} [Note spelling: tessellation]
In a nutshell, every edge of the tessellation 
represents a polymer chain (the vertices are then permanent cross-links). 
As argued in \cite{Gloria-LeTallec-Vidrascu-08b} and \cite{Alicandro-Cicalese-Gloria-07b},
it is reasonable at first order to consider cross-links at zero temperature and the polymer chains at finite temperature.
The free energy of the $\e$-rescaled network of polymer chains in a domain $D$ of $\R^3$ is then given by the sum of the free energies of the edges,
plus a volumetric term accounting for the incompressibility of the network (the volume of each simplex is conserved).
Under some assumptions on the energy terms (which are standard in the theory of homogenization of integral functionals), and under
assumptions on the stochastic lattice compatible with the random parking measure, it is proved in \cite{Alicandro-Cicalese-Gloria-07b} 
that the free energy functional of the $\e$-rescaled network in $D$ (seen as a function of the position of the cross-links)
$\Gamma$-converges (as $\e$ goes to zero) to a continuous energy functional of the type 
\begin{equation*}
u\mapsto \int_D W_\ho(\nabla u(x))\ud x,
\end{equation*}
where the ``homogenized energy density'' $W_\ho$ is a quasiconvex frame-invariant function, as encountered in continuum mechanics.
The map $u:D\to \R^3$ is a deformation.
In particular, this type of convergence ensures that minimizers of the free energy of the discrete system converge (up to extraction)
to minimizers of the continuous energy functional --- which yields a rigorous derivation of nonlinear elasticity compatible
with minimization.
In addition, if the stochastic lattice is statistically isotropic, then the associated homogenized energy density $W_\ho$ is isotropic, as expected
in rubber elasticity.
The stochastic lattices considered in \cite{Alicandro-Cicalese-Gloria-07b}
satisfy the following three properties: ergodicity,
non-empty space condition, and hard-core condition.
Two questions were left open in \cite{Alicandro-Cicalese-Gloria-07b}:
\begin{itemize}
\item Do there exist such stochastic lattices which are indeed statistically isotropic ?
\item What happens to the thermodynamic limit in $D$ if the stochastic lattice on $\R^d$
is replaced by an approximation on $D$ (in which case boundary effects may appear
and rule out stationarity) ?
\end{itemize}
This article gives a clear answer to both questions.
As we quickly show in Subsection~\ref{sec:parking}, the random parking measure in $\R^d$ is an example
of stochastic lattice which is statistically isotropic, ergodic, and satisfies the non-empty and hard-core conditions.
In Subsection~\ref{sec:ass-SET}, we shall prove that the model studied in \cite{Alicandro-Cicalese-Gloria-07b} can be recast in terms 
of a subadditive function satisfying the assumptions introduced above, 
so that the identity corresponding to \eqref{1:eq:sub-limit} will hold true.
This allows us to complete the program initiated in \cite{Alicandro-Cicalese-Gloria-07b} both in terms 
of lattices generated in $D$ instead of $\R^d$ (which seems more reasonable from a physical point of view),
and in terms of numerical approximations (effective computations for $W_\ho$ in \cite{Gloria-LeTallec-Vidrascu-08b} 
are based on random parking in bounded domains $Q_R$). 

\medskip
This paper is organized as follows. Section~\ref{sec:parking-SET} is dedicated to the study of the
qualitative properties of the random parking measure, and to the proof of \eqref{1:eq:sub-limit}.
In Section~\ref{sec:TSP} we apply the main result to classical Euclidean optimization problems.
Finally, Section~\ref{sec:rubber-thermo} is concerned with the application of the main result to the discrete model 
for rubber studied in \cite{Alicandro-Cicalese-Gloria-07b}.
In particular, we show that this model satisfies the assumptions of the main result, which completes
the analysis of the model when the stochastic lattice is the random parking measure.

\medskip
We make use of the following notation throughout the article:
\begin{itemize}
\item $d,n\geq 1$ denote dimensions;
\item $\R^+=[0,+\infty)$;
\item $\calO(\R^d)$ is the set of bounded nonempty
Lipschitz subsets of $\R^d$ (a bounded measurable set $D$, whose boundary is denoted by $\partial D$, is
Lipschitz if there exists a finite collection of relatively
open sets $U_k \subset \partial D$ such that
$\partial D = \cup_k U_k$, and $U_k$ is the graph of a Lipschitz function
on $\R^{d-1}$ up to a rotation);
\item For all $D\in \calO(\R^d)$, $|D|$ denotes its $d$-dimensional Lebesgue measure, $\partial D$ denotes its boundary, and $\overline{D}$ its closure;
%and $\mathring{D}$ its interior;
\item For all $D\in \calO(\R^d)$ we set 
$\calO(D)=\{A\in \calO(\R^d) \big| A\subset D\}$ (note that $D\in \calO(D)$);
% \item For all $x\in \R^d$, $|x|_\infty=\max\{|x_1|,\dots,|x_d|\}$, 
% and $d_\infty$ denotes the supremum-norm distance, i.e. 
% $d_\infty(x,A) := \inf \{|x-y|_\infty: y \in A\}$
% for nonempty $A \subset \R^d$;
\item For all $x\in \R^d$, $|x|$ is the Euclidean norm of $x$,
and $d_2$ denotes the Euclidean distance, i.e. 
$d_2(x,A) := \inf \{|x-y|: y \in A\}$ for nonempty $A \subset \R^d$. 
%\Comment{we don't use $d_\infty$ anymore}
\item $\diam$ denotes the Euclidean diameter of a subset of $\R^d$;
\item $\M^{n\times d}$ is the space of $n\times d$ matrices, that we simply denote by $\M^{d}$ when $n=d$;
\item For all $\Lambda \in \M^{n\times d}$, we define $\varphi_\Lambda:\R^d\to \R^n$ by $x\mapsto \Lambda x$, and denote by $|\Lambda|$ the Frobenius norm  $\sqrt{\mbox{Trace}(\Lambda^T :\Lambda)}$ (which is the operator norm of $\varphi_\Lambda$ associated with the Euclidean norms of $\R^d$ and $\R^n$);
\item $SO_d$ is the set of rotations in $\R^d$;
\item $\calA_{f}^d$ and $\Set^d$ denote the spaces of finite and locally finite point sets
in $\Rr^d$, respectively;
%
%\item For all $D\in \calO(\R^d)$, $C(D)$ and $C^\infty_0(D)$ respectively
% denote the space of continuous functions on $D$, and the space of smooth 
%functions which vanish identically
%on a neighbourhood of $\partial D$;
\item For all open $D\subset \R^d$, $C(D,\R^n)$ and $C^\infty_0(D,\R^n)$ respectively
denote the space of continuous functions from $D$ into $\R^n$,
and the space of smooth 
functions which have support in $D\setminus \partial D$;
\item For all $D\in \calO(\R^d)$ open, $n\geq 1$, and $p\in [1,\infty]$, $L^p(D,\R^n)$, $W^{1,p}(D,\R^n)$,
and $W^{1,p}_0(D,\R^n)$ denote the Lebesgue space of
$p$-integrable (or essentially
bounded if $p=\infty$) functions from $D$ into $\R^n$,
the Sobolev space of $p$-integrable
functions on $D$ whose distributional derivatives are $p$-integrable (or 
essentially bounded if $p=\infty$), 
and the closure of $C^\infty_0(D,\R^n)$ in  $W^{1,p}(D,\R^n)$, respectively;
% \Comment{added 'open'+values in $\R^n$}
\item When no confusion occurs we use the short-hand notation   $C(D)$, $L^p(D)$ etc. for  $C(D,\R^n)$ in  $L^{p}(D,\R^n)$ etc.;
%\Comment{new def}
\item For all $D\in \calO(\R^d)$ open and $u \in W^{1,\infty}(D)$, we denote by $\|u\|_\Lip$ the Lipschitz constant of $u$ on $D$.
\end{itemize}
%

%%%%%%%%%%%%%%%%%%%%%%%%%%%%%%%%%%%%%%%%%%%%%%%%%%%%%%%%%%%%%%%%%%%%%%%%%%
%%%%%%%%%%%%%%%%%%%%%%%%%%%%%%%%%%%%%%%%%%%%%%%%%%%%%%%%%%%%%%%%%%%%%%%%%%

\section{Random parking and subadditive ergodic theorems}\label{sec:parking-SET}

\subsection{Random parking}\label{sec:parking}

We first quickly recall the graphical construction of the random parking measure $\xi^A$ in a
%n arbitrary
%some
Borel set $A\subset \R^d$.
Let $\calP$ be a homogeneous Poisson process of unit intensity in $\R^d\times \R^+$.
An \emph{oriented graph} is a special kind of directed graph in which there is no pair of vertices $\{x,y\}$ 
for which both $(x,y)$ and $(y,x)$ are included as directed edges. We shall say that $x$ is a \emph{parent} of $y$ 
and $y$ is an \emph{offspring} of $x$ if there is an oriented edge from $x$ to $y$. By a \emph{root} of an oriented
graph we mean a vertex with no parent.

\medskip
The graphical construction goes as follows. Let $\rho_0>0$ and let $B$ denote the Euclidean ball in $\R^d$ of radius
$\rho_0$ centred at the origin. Make the points of the Poisson process $\calP$
on $\R^d\times \R^+$ into the vertices of an infinite oriented graph, denoted by $\calG$, by putting in an oriented
edge $(X,T)\to (X',T')$ whenever $(X'+B)\cap(X+B)\neq \emptyset$ and $T<T'$. For completeness we also put an edge 
$(X,T)\to (X',T')$ whenever $(X'+B)\cap(X+B)\neq \emptyset$, $T=T'$, and $X$ precedes $X'$ in the lexicographical order --- although
in practice the probability that $\calP$ generates such an edge is zero.
It can be useful to think of the oriented graph as representing the spread of an ``epidemic'' through space over time;
each time an individual is ``born'' at a Poisson point in space-time, it becomes (and stays) infected if there is an earlier infected
point nearby in space (in the sense that the translates of $B$ centred at the two points overlap).
This graph determines which items have to be accepted.

\medskip
For $(X,T)\in \calP$, let $C_{(X,T)}$ (the ``cluster at $(X,T)$'') be the (random) set of ancestors of $(X,T)$, that is, the set of 
$(Y,U)\in \calP$ such that there is an oriented path in $\calG$ from $(Y,U)$ to $(X,T)$. 
As shown in \cite[Corollary~3.1]{Penrose-01}, the ``cluster''  $C_{(X,T)}$  is finite for $(X,T)\in \calP$ with probability 1.
It represents the set of all items that can potentially affect the acceptance status of the incoming particle represented by the Poisson
point $(X,T)$.
The 
%method of recontructing the 
set of accepted items may be reconstructed from the graph $\calG$,
as follows. Let $A\subset \R^d$ be a (possibly unbounded)
Borel set, and let $\calP_A$  denote the set $\calP \cap (A\times \R^+)$, i.~e. the set of Poisson points that lie in $A\times \R^+$.
Let $\calG_{|A}$ denote the restriction of $\calG$ to the vertex set $\calP_A$.
Recursively define subsets $F_i(A)$, $G_i(A)$, $H_i(A)$ of $A$, $i=1,2,3,\dots$ as follows.
Let $F_1(A)$ be the set of roots of the oriented graph $\calG_{|A}$, and let $G_1(A)$ be the set 
of offspring of roots. Set $H_1(A)=F_1(A)\cup G_1(A)$. For the next step, remove the set $H_1(A)$
from the vertex set, and define $F_2(A)$ and $G_2(A)$ the same way; so $F_2(A)$ is the set of roots of the restriction
of $\calG$ to vertices in $\calP_A \setminus H_1(A)$, and $G_2(A)$ is the set of vertices in $\calP_A \setminus H_1(A)$
which are offspring of $F_2(A)$. Set $H_2(A)=F_2(A)\cup G_2(A)$, remove the set $H_2(A)$
from  $\calP_A \setminus H_1(A)$, and repeat the process to obtain $F_3(A)$, $G_3(A)$, $H_3(A)$.
Continuing ad infinimum gives us subsets $F_i(A)$, $G_i(A)$ of $\calP_A$ defined for $i=1,2,3,\dots$.
These sets are disjoint by construction. In the case when $A=\R^d$, we drop the reference to $A$
and use the abbreviation $F_i$ and $G_i$ for $F_i(\R^d)$ and $G_i(\R^d)$.

\medskip
As proved in \cite[Lemma~3.2]{Penrose-01}, for every bounded nonnull Borel set $A$ in $\R^d$,
the sets $F_1(A)$, $G_1(A)$, $F_2(A)$, $G_2(A),\dots$ form a partition of $\calP_A$, and 
$F_1$, $G_1$, $F_2$, $G_2,\dots$ form a partition of $\calP$
with probability 1. 

\begin{defi}\label{defi:RP}
The random parking measure $\xi^A$ in $A$ is given by the projection of the union $\cup_{i=1}^\infty F_i(A)$ on $\R^d$.
Likewise, the random parking measure $\xi$ in $\R^d$ is given by the projection of the union $\cup_{i=1}^\infty F_i$ on $\R^d$.
\end{defi}

% \medskip
% %
% In what follows for all $t>0$ we shall also need another point process $\xi_t^A$, defined
% by a similar graphical construction $(\calG_t)_{|A}$ given by the restriction of $\calG$ to the vertices
% $(X,T)$ of $\calP_A$ such that $T<t$.
% As shown in \cite[Theorem~2.2]{Penrose-01}, for all $t \in (0,+\infty]$, $\xi_t^A$ weakly converges
% in the sense of measures to $\xi^A$ as $t$ goes to infinity, 
% and weakly converges in the sense of measures to $\xi_t$ as $A$ tends to $\R^d$.

% [Definition of admissible moved earlier than before (since used here).
% Also changed definition
% \Comment{MP} 
% style so definitions easier for the reader to find.
% Also a point set in $\R^d$ is not an element of
% $(\R^d)^{\N}$ because one does not care about ordering
% or allow repeats. Finally, I dropped `general' from
% being part of the definition of `admissible'. Can reinstate
% if needed.]

\begin{defi}\label{def:general}
%Let $\Sigma\in (\R^d)^{\N}$
Let $\Sigma \subset \R^d$
be a locally finite set of points. We say that
$\Sigma$ is {\em general} if no $d+1$ points lie in the same hyperplane
and if no $d+2$ points lie in the same hypersphere.
\end{defi}
%
%We now formally define the class of admissible point sets as those
% satisfying the non-empty space and hard-core conditions:

\begin{defi}
Let $0 < \rho_1 \leq \rho_2 \leq \infty$.
Given $D \subset \R^d$,
% and $I$ be a subset of $\N$.
suppose $\Sigma$ is a subset of $D$.
%A general set of points $\Sigma\in D^I$
Then $\Sigma$ 
is said to be $(\rho_1,\rho_2)$-{\em admissible in $D$}
iff for all $x\neq y\in \Sigma$, $|x-y|\geq \rho_1$, and for all $z\in D$, $B(z,\rho_2)\cap \Sigma\neq \emptyset$,
where $B(z,\rho_2)$ 
denotes the open ball centred in $z$ of radius $\rho_2$
(and $B(z,\infty):= \R^d)$).
Let $\mathcal{A}_{\rho_1,\rho_2}$ denote
the class of
general point sets that are 
$(\rho_1,\rho_2)$-admissible
in $\R^d$.
\end{defi}
In other words, an admissible point set in $D$ is one which
satisfies the hard-core and empty space conditions.
We shall sometimes write simply `admissible' for
`admissible in $\R^d$'.

\begin{defi}\label{def:iso1}
A random subset $\Sigma$ of $\R^d$ is called a {\em point process}.
We say that a point process $\Sigma$ in $\R^d$ 
% stochastic set of points $\Sigma$ taking values in $(\R^d)^{\N}$ 
is {\em isotropic} if the distribution of $\Sigma$ and of
$\mathbf{R}\Sigma$ are the same
for every rotation $\mathbf R \in SO_d$.
\end{defi}
%
%
%\medskip
%
We are now in position to prove that the random parking measure $\xi$ in $\R^d$ (which is admissible by construction)
is an isotropic admissible stochastic lattice in the sense of \cite{Alicandro-Cicalese-Gloria-07b}.

% [Do we want to include the statement that
% \Comment{MP} 
% $\xi$ is admissible in this result?]

%
%
%
\begin{proposition}\label{prop:iso-ergo}
The random measure $\xi$ is stationary (under real shifts), ergodic, isotropic, and 
almost surely general.
\end{proposition}
\begin{proof}[Proof of Proposition~\ref{prop:iso-ergo}]
We prove each property separately.

\step{1}{Stationarity}
By definition, the Poisson point process $\calP$ on $\R^d\times \R^+$  and its translation $(x_1,\dots,x_d,0)+\calP$ 
have the same distribution for all $x=(x_1,\dots,x_d)\in \R^d$. Hence the graphical construction is stationary, and therefore also the 
random parking measure $\xi$.

\step{2}{Isotropy}
The proof of the isotropy of $\xi$ is similar to the proof of the stationarity. For all $\mathbf{R}\in SO_d$, the Poisson point process $\calP$ 
and its rotated version $\mathbf{R}\calP\,:=\,\{(\mathbf{R}x,t)\,:\, x\in \R^d, t\in \R^+, (x,t)\in \calP\}$
have the same distribution, which implies that the random parking measure $\xi$ is isotropic.

\step{3}{General position}
For all $t>0$ we set $\overline{\calP}_t=\{x\in \R^d\,:\,\exists \tau \in [0,t], (x,\tau)\in \calP\}$.
Since $\xi \subset \cup_{n\in \N}\overline{\calP}_n$, and since $\overline{\calP}_n\subset \overline{\calP}_{n+1}$ for all $n\in \N$,
the event that $\xi$ is not in general position is contained in the union over
$n\in \N$ of the events that $\overline{\calP}_n$ is not
in general position.
Since $\overline{\calP}_n$ is a Poisson process of intensity $n$ in $\R^d$, the probability that $\overline{\calP}_n$ is not in general position is zero, so that
the union of this countable set of events has also probability zero, and $\xi$ is almost surely general.
% To prove that $\xi$ is almost surely general, it is enough 
% to show that the point process $\overline{\calP}:=\{x\in \R^d\,:\,\exists t\in \R^+, (x,t)\in \calP\}$
% is almost surely general, since $\xi$ is a subset of $\overline{\calP}$.
% To this aim, for all $R\in \N$ and $t\in \N$, we define the point
% process $\calP_{t,R}:=\{t^{1/d}x\in \R^d\,:\, |x|\leq R,\, \exists \tau\leq t, (x,\tau)\in \calP\}$.
% By definition, $\calP_{t,R}$ is a  Poisson point process in $(-t^{1/d}R,t^{1/d}R)^d$ of the same intensity as $\calP$,
% so that the set of events $\Omega_{t,R}$ for which $\calP_{t,R}$ is not general is a null set.
% Hence, the set of events such that $\overline{\calP}$ is not general --- which is contained in $\cup_{t\in \N,R\in \N}\Omega_{t,R}$ ---
% has also probability zero.

\step{4}{Ergodicity}
Let us view $\xi$ and $\Po$  as
elements of $\Set^d$ and $\Set^{d+1}$ respectively.
Also let us extend $\Po$
to a homogeneous Poisson process of unit intensity on the whole of $\Rr^{d+1}$
(also denoted $\Po$). Let $T_x$ denote translation
by an element $x$ of $\Rr^d$ (acting either on $\Set^d$ or
$\Set^{d+1}$ according to context).
Then as described earlier in this section,
$ \xi$ is the image of $\Po$ under a certain
mapping $ h$ from $\Set^{d+1}$ to $\Set^d$, which commutes
with $T_x$ for any $x \in \Rr^d$, that is $T_x \circ h = h \circ T_x$.

Suppose $A$ is a measurable subset of $\Set^d$ which is shift-invariant,
meaning $T_x(A)=A$ for all $x \in \Rr^d$. Then for $x \in \R^d$,
$$
h^{-1} (A) = h^{-1}(T_x (A)) = T_x (h^{-1} A )
$$
so $h^{-1}(A)$ is invariant under the mapping $T_x$ acting
on $\Set^{d+1}$.
Now, $\Po$ is ergodic under translations,
that is if $B \subset \Set^{d+1}$ satisfies $T_x(B)=B$ for
some non-zero $x \in \R^{d+1}$, then $P[B]\in \{0,1\}$. 
See for example the proof of 
Proposition 2.6 of \cite{Meester-Roy96}.
Therefore with $A$ as above, 
$$
P [ \xi \in A ] = P [ \Po \in h^{-1}(A) ]  \in \{ 0,1 \}.
$$
Thus $\xi$ is ergodic.
\end{proof}

\subsection{Main result}\label{sec:result}

For $i \in \Z^d$, let $C^i := i+[-1/2,1/2]^d$ be the closed rectilinear unit
cube centred at $i$.

Let $0 < \rho_1 < \rho_2$.
Given $D\in \calO(\R^d)$, 
and  $R>0$, let
$D_{R}:=\{Rx, x\in D\}$ be the dilation of $D$  by a factor $R>0$.
%Let $\calD(D) : = \{D_R: R >0\}$ be the set of all such dilations. 
For $r >0$ set $D_{R,r}:= \{ x \in D_R: 
d_2(x, \R^d \setminus  D_R) > r)$.
%Let $D^*_{R,r}$ be the union of all cubes of the form
%$C^i, i \in \Z^d$ that are contained in $D_{R,r}$.
Also set $D_{R,0} := D_R$, and 
let $\calD(D) : = \{D_{R,r}: R >0, r \geq 0, 
D_{R,r} \in \calO(\R^d) \}$.
The next result shows that $D_{R,r} \in \calD(D)$ 
whenever $r/R$ is small enough.

\begin{lemma}
\label{lemlip}
Suppose $D \in \calO(\R^d)$. Then there exists $\delta >0$
such that $D_{1,t} \in \calO(\R^d)$
for all  $t \in ]0 ,\delta[$.
\end{lemma}
\begin{proof}
Assume without loss of generality 
that $d \geq 2$ and $D$ is closed.
For $r >0$ let $B'_r$ denote the closed Euclidean ball of
radius $r$ centred at the origin in $\R^{d-1}$. 
Suppose $x \in \partial D$. By the definition of $\calO(\R^d)$
we may assume after translation and rotation that $x =0$ and
that there exist $r >0$ and $a >0$, and Lipschitz
function $g: B'_{2r} \to \R$  such that 
$|g(u)| \leq a$ for all $u \in B'_{2r}$ and
$$
D \cap ( B'_{2r} \times [-4a,4a])  = \{
(u,y)| u \in B'_{2r}, y \in  [-4a ,g(u)] \}.
$$   
Suppose $0 < t < \min(r,a)$. Then
$$
D_{1,t} \cap 
( B'_{r} \times [-2a,2a])
= \{(u,y)| u \in B'_{r},
-2a  \leq y < \inf_{w \in B'_{t} }  g_{w}(u) \}
$$
where for $w \in B'_t$
we set $g_w(u) = g(u+w) - \sqrt{t^2 - |w|^2 }$.
Since the infimum of Lipschitz functions 
with common Lipschitz constant
is itself Lipschitz,
it follows that the intersection of $\partial D_{1,t}$
with the interior of $ B'_{r} \times [-2a,2a]$ is
the graph of a Lipschitz function of $d-1$ variables.
We may then use a compactness argument to deduce
that there exists $\delta >0$ such that for
$0<t<\e$ the set $D_{1,t}$ 
is Lipschitz.
\end{proof}
It follows from Lemma \ref{lemlip} that
given any $D \in O(\R^d)$ and $0 < \alpha < 1$,  the
set $D_{R,R^\alpha}$ is in $\calO(\R^d)$ for all large enough $R$.
%
%
%For all $0<\alpha<1$ and $D$, $D_{R,R^\alpha}$ has exactly one
% connected component for $R$ large enough.
%In particular for $R$ large enough, $D_{R,R^\alpha}\in \calO(\R^d)$. 
%Here comes the argument: $\partial D_R$ is locally the graph of a Lipschitz
% function from $\R^{d-1}$ to $\R$, so that $\partial D_{R,R^\alpha}$ is
% locally the graph of the maximum (or minimum, depending on which side of
% the graph $D_R$ lies) of the graphs of $d$ Lipschitz functions (the
% $d$ functions are the original Lipschitz function shifted by unit vectors). 
%Since the maximum of $d$ Lipschitz functions remains Lipschitz,  
%$\partial D_{R,R^\alpha}$ is locally
%the graph of a Lipschitz function.
%
We now define some properties of functions parametrized by point sets 
(or restrictions of point sets on bounded domains).
\begin{defi}\label{defi:sub}
Let $D \in \calO(\R^d)$.
A measurable function 
$S:\calD(D) \times \calA_{\rho_1,\rho_2}\to \R$ 
is said to be:
\begin{itemize}
\item {\em local on $\calD(D)$}
if for all $\hat D \in \calD(D)$,
%$R>0$, $r \geq 0$
% $D\in \calO(\R^d)$,
and all $\zeta,\tilde \zeta \in \calA_{\rho_1,\rho_2}$ such that 
$\zeta \cap \hat D =\tilde \zeta \cap \hat D $,
we have
\begin{equation*}
S(\hat D,\zeta)\,=\,S(\hat D,\tilde \zeta);
\end{equation*}
\item 
%(insensitivity to boundary effects)
{\em insensitive to boundary effects on $\calD(D)$} if
there exists $0<\alpha<1$ such that 
we have
\begin{equation}
\limsup_{R\to +\infty}\sup_{\zeta \in \calA_{\rho_1,\rho_2}}
\left\{\frac{|S(D_R,\zeta)-S(D_{R,R^\alpha},\zeta)|}{R^d}\right\}\,=\,0
\label{insensalpha}
\end{equation}

\end{itemize}
\begin{itemize}
\item
Also,
$S$ is said to have  the {\em averaging property} on $\calA_{\rho_1,\rho_2}$ 
with respect to $D$ if
for any stationary and ergodic random lattice $\zeta$ whose realization almost surely belongs to
$\calA_{\rho_1,\rho_2}$, 
%and is almost surely general,
there exists a deterministic $\overline S=\overline S(\zeta) \in \R$
such that almost surely
\begin{equation}
\lim_{R\to +\infty}\frac{S(D_{R},\zeta)}{|D_{R}|}\,=\,\overline{S}.
\label{avging}
\end{equation}
\end{itemize}
\end{defi}
If $S$ is local, then we can properly define $S(D,\zeta)$ 
for a point set $\zeta$ only defined in $D\in \calO(\R^d)$, provided it
is general and $(\rho_1,\rho_2)$-admissible in $D$.
We do this as follows:
let $\zeta^*$ be a choice of 
general and
$(\rho_1,\rho_2)$-admissible set in $\R^d$ such that $\zeta^* \cap D = \zeta$,
and set $S(D,\zeta) := S(D,\zeta^*)$.
By locality this definition does not depend on the choice
of $\zeta^*$, and there does exist at least one such choice
of $\zeta^*$ because if we choose $\rho_0 \in (\rho_1,\rho_2)$
and perform the random parking
process in $\R^d$  as described in section \ref{sec:parking}, but 
rejecting any point that lies within distance $\rho_0$ of
any point of $\zeta$, then the resulting set of accepted points,
together with the point set $\zeta$ itself, will be $(\rho_1,\rho_2)$-admissible
on the whole of $\R^d$, and almost surely general.
%[Expanded this paragraph \Comment{MP}]

% satisfies the hard-core and non-empty space 
%conditions with constants $(\rho_1,\rho_2)$ in $D$.

\medskip
In Sections \ref{sec:TSP} and \ref{sec:rubber-thermo}
we shall give two examples --- or rather two classes of functionals --- which satisfy the above
properties.  Both examples exhibit subadditivity, and we
shall make use of the Akcoglu-Krengel subadditive
ergodic theorem to show they satisfy the averaging property.

\medskip
Let $0 < \rho_1 < \rho_0 < \rho_2 < \infty$, 
and
for $D\in \calO(\R^d)$, 
let $\xi$ (respectively $\xi^D$)  be the random parking measure on 
$\R^d$ (respectively on $D$) with parameter $\rho_0 $
%For $D\in \calO(\R^d)$, let
%$\xi^{D_R}$ denote the random parking measure on $D_R=\{Rx, x\in D\}$
(as in Definition \ref{defi:RP}).
The main result of this section is the following:
%[Dropped Unif bdd condition in this result \Comment{MP}]
%
\begin{theorem}\label{th:sub}
Let $D\in \calO(\R^d)$.
% and 
%for all $R>0$, 
%let $\xi^R=
%$\xi^{D_R}$ denote the random parking measure on $D_R :=\{Rx, x\in D\}$,
If the measurable function $S:\calD(D) 
\times \calA_{\rho_1,\rho_2}\to \R^+$ is local on $\calD(D)$,
insensitive to boundary effects on $\calD(D)$,
and has the averaging property on $\calA_{\rho_1,\rho_2}$
with respect to $D$,
then with $\overline{S}$ given by \eq{avging},
almost surely
\begin{equation}
\label{av2}
\lim_{R\to +\infty} \frac{S(D_R,\xi^{D_R})}{|D_R|}\,
=\,\lim_{R\to +\infty} \frac{S(D_R,\xi)}{|D_R|}\,=\,\overline{S}.
\end{equation}
\end{theorem}
Theorem~\ref{th:sub} can be seen as a version of the `subadditive ergodic theorem' for
random parking measures on homothetic sets. It is proved 
using the following result due to Schreiber, Yukich and
the second author \cite{SPY},
which gives an `exponential stabilization' property
of $\xi$. Let us say $\calY \subset \R^d \times \R^+$
is {\em temporally locally finite} if
$\calY \cap (\R^d \times [0,t])$ is finite
for all $t \in (0,\infty)$. 
Also, for $r >0$ let $B_r$  denote the Euclidean ball
$\{x \in \R^d:|x| \leq r\}$.
%\Comment{MP}
%{\bf [New lemma. Check notation.]}

\begin{lemma}\label{lem:expstab}
There exists a nonnegative random variable $\RS$ 
such that
(i) $\RS$ has an exponentially decaying tail,
i.e. for some constant $c>0$ the probability that
$ \RS $ exceeds $ t$ is at most $c \exp(-c^{-1}t)$ for all $t \geq 0$;
%
%i.e. $ \limsup_{t \to \infty} t^{-1} \log P[R^* \geq t]  < 0$,
%
and (ii) for all temporally locally finite
$\calY \subset (\R^d \setminus B_\RS) \times \R^+ $
the measures 
$\xi^{B_\RS}$ and $\xi^{B_\RS , \calY}$
coincide on $[-1/2,1/2]^d$,
where $\xi^{B_\RS , \calY}$ 
is the parking measure induced by the union
of  $\Po \cap ( B_\RS \times \R^+)$ and $\calY$. 
\end{lemma}
\begin{proof}
Essentially this result is Lemma 3.5 of \cite{SPY}.
That result  says only that
the total number of points in $[-1/2,1/2]^d$
coincide for $\xi^{B_\RS , \calY}$ and $\xi^{B_\RS}$,
not that the measures themselves coincide on
$[-1/2,1/2]^d$.  However, the proof in \cite{SPY}
in fact demonstrates the stronger statement given here.
The proof in \cite{SPY} assumes that $\rho_0 <1/4$ (see \cite{SPY}, p.173)
but once we have the result for this case, 
we can deduce it for the general case by a
scaling argument.
\end{proof}

In the next lemma, we use Lemma \ref{lem:expstab} to
show that  $S(D_{R,R^\alpha},\xi)$ is a good
approximation to  $S(D_{R,R^\alpha},\xi^{D_R})$.
\begin{lemma}
\label{e2lemma}
Suppose $D \in \calO(\R^d)$ and $S: \calD(D) \times \calA_{\rho_1,\rho_2}
\to \R^+$ is local on $\calD(D)$.
Suppose
$0 < \alpha < 1$.
Then there exists an almost surely finite random variable
$R_0$ such that for $R \geq R_0$ we have
$S(D_{R,R^\alpha},\xi^{D_R}) = S(D_{R,R^\alpha},\xi)$.
\end{lemma}
\begin{proof}
For $i \in \Z^d$ and $r>0$, let
$B^i_{r}$ denote the closed Euclidean ball of radius $r$ centred at $i$,
and let $\overline{Q}_{1/2}^i$ denote
the rectilinear unit cube centred
at $i$.
By Lemma \ref{lem:expstab}, there exists
a family of random variables
$(\RS(i),i \in \Zz^d)$
which are identically distributed with exponentially
decaying tails, such that
$\xi^{B^i_{\RS(i)} , \calY}$
coincides with $ \xi^{B^i_{\RS(i)}}$ on ${\overline Q_{1/2}^i}$,
for all temporally locally finite $\calY \subset
(\R^d \setminus B^i_{\RS(i))}) \times \R^+$, where
$\xi^{B^i_{\RS(i)} , \calY}$ denotes
the parking measure induced by the union
of  $\Po \cap ( B^i_{\RS(i)} \times \R^+)$ and $\calY$. 

Let $D^*_{R,R^\alpha} := \{ i \in \Z^d: \overline Q_{1/2}^i \cap D_{R,R^\alpha}
\neq \emptyset \}$.
For $k \in \N$, let $E(k)$ denote the event that there exists
$R \in [k,k+1[$ and 
$i \in D^*_{R,R^\alpha} $ such that
$B^i_{\RS(i)} \setminus D_R \neq \emptyset$. If
$E(k)$ does not occur, then for all 
$R \in [k,k+1[$ and all
%$i \in \Z^d \cap D^*_{R,R^\alpha} $
$i \in  D^*_{R,R^\alpha} $
we have  $B^i_{\RS(i)} \subset D_R$, and therefore
both $\xi$ and $\xi^{D_R}$ coincide with 
$ \xi^{B^i_{\RS(i)}}$ on ${\overline Q_{1/2}^i}$. Hence
$\xi$ and $ \xi^{D_R}$ coincide on the whole of $D_{R,R^\alpha}$.

Choose $K$ such that $D \subset B_K$.
Suppose $k \in \N$ and $R \in [k,k+1[$.
Then 
for $i \in D^*_{R,R^\alpha}$,
since $i$ is distant at least $R^\alpha -d$
from $\R^d \setminus D_R$ by the definition
of $D^*_{R,R^\alpha}$, we have that
$B_{k^\alpha-d}^i \subset B_{R^\alpha-d}^i \subset D_R$.
%$$
%|x-y|_\infty \geq |x - (R/k)y|_\infty - |(R/k)y -y|_\infty 
%\geq k^\alpha - C. 
%$$
Also, we have 
$D^*_{R,R^\alpha}  \subset B_{KR+d} \subset B_{K(k+1)+d}$.

Therefore, for large enough $k \in \N$
the probability that $E(k)$ occurs is bounded
by the cardinality  of $ \Z^d \cap B_{K(k+1)+d} $, 
multiplied by the probability that $\RS(0) \geq k^{\alpha} -d $,
and by Lemma \ref{lem:expstab}  this is bounded
by $ck^d \exp(-c^{-1} k^\alpha)$ for
some universal constant $c >0$. 
Hence, the probability of $E(k)$ is summable over
$k\in \N$, and from the Borel-Cantelli lemma
we deduce that almost surely, $E(k)$ occurs for only
finitely many $k$ so that there exists $R_0$ such that
for all $R \geq R_0$ the measures
$\xi$ and $\xi^{D_R} $ coincide  on $ D_{R,R^\alpha}$,
so that $S(D_{R,R^\alpha},\xi)= S(D_{R,R^{\alpha}},\xi^{D_R})$
by locality. 
\end{proof}

\begin{proof}[Proof of Theorem~\ref{th:sub}]
Since $\xi$ is ergodic by Proposition~\ref{prop:iso-ergo}, 
and $S$ has the averaging property in the sense
of Definition~\ref{defi:sub}, there exists $\overline{S}\in \R^+$
such that \eq{avging} holds, i.e. the second equality of \eq{av2} holds.
Choose $\alpha \in (0,1)$ such that \eq{insensalpha} holds.
To prove the first equality of \eq{av2},
for all $R\geq 1$, we write $\xi^R$ for $\xi^{D_R}$
and rewrite the averaged energy as
\begin{equation}
\label{0822a}
\frac{S(D_R,\xi^R)}{|D_R|}\,=\,
\frac{S(D_{R},\xi)}{|D_{R}|}
+e_1(R)+e_2(R)+e_3(R),
\end{equation}
with
\begin{eqnarray*}
e_1(R)&:=& \frac{S(D_{R,R^\alpha},\xi)}{|D_{R}|}-
\frac{S(D_{R},\xi)}{|D_{R}|},
\\
e_2(R)&:=&
\frac{S(D_{R,R^\alpha},\xi^R)}{|D_{R}|}-\frac{S(D_{R,R^\alpha},\xi)}{|D_{R}|},
\\
e_3(R)&:=&\frac{S(D_{R},\xi^R)}{|D_{R}|}-
\frac{S(D_{R,R^\alpha},\xi^R)}{|D_{R}|}.
\end{eqnarray*}
Now $e_2(R)$ tends to zero almost surely by Lemma \ref{e2lemma},
and by the insensitivity to boundary effects 
$e_1(R)$ and $e_3(R)$ 
tend to zero almost surely as $R \to \infty$,
so that the first equality of \eq{av2} follows.
\end{proof}

\section{Application to classical Euclidean optimization problems}\label{sec:TSP}

A measurable functional $S:\calO(\R^d)\times \calA_{lf}^d \to \R$ is
said to be {\em translation invariant} if
for every $y\in \R^d$, $D\in \calO(\R^d)$, and every $\zeta \in \calA_{lf}^d$, 
we have $S(D+y,\zeta+y)\,=\,S(D,\zeta)$,
where for any subset $U$ of $\R^d$, $U+y\,:=\,\{z+y\,:\,z\in U\}$.
\begin{defi}
Let $S:\calO(\R^d)\times \calA_{lf}^d\to \R$ be a translation
invariant measurable
%Euclidean
functional, and let $p\geq 0$.
We say that $S$ is 
\begin{itemize}
\item {\em localized} if there exists $\tilde S:\calA_{f}^d \to \R$ such that for all $\zeta \in \calA_{lf}^d$ and $D\in \calO(\R^d)$,
$S(D,\zeta)\,=\,\tilde S(\zeta \cap D)$,
\item {\em almost subadditive} of order $p$ if there exists $C>0$ such that for all $\zeta_1,\zeta_2 \in \calA_{lf}^d$ and
all $D\in \calO(\R^d)$,
\begin{equation*}
S(D,\zeta_1\cup \zeta_2)\,\leq \, S(D,\zeta_1)+S(D,\zeta_2)+C\diam(D)^p;
\end{equation*}
\item (strongly) {\em superadditive} on rectangles if for any partition of a semi-open rectangle $D$ into
a finite number $k$ of semi-open rectangles $\{D_i\}_{i\in\{1,\dots,k\}}$,
\begin{equation*}
S(D,\zeta)\,\geq \, \sum_{i=1}^k S(D_i,\zeta);
\end{equation*}
\item {\em smooth} of order $p$ if 
%\Comment{MP}
%{\bf [I think we can we get away with cubes rather than rectangles ]}
there exists $C>0$ such that for every rectilinear
cube $D$ and all $\zeta_1,\zeta_2 \in \calA_{lf}^d$,
\begin{equation}
|S(D,\zeta_1\cup \zeta_2)-S(D,\zeta_1)|\,\leq\, C\diam(D)^p(\card \zeta_2\cap D)^{(d-p)/d}.
\label{smooth0}
\end{equation}
\item {\em homogeneous} of order $p$
if for all $\alpha>0$, $D\in \calO(\R^d)$, and  $\zeta \in \calA_{lf}^d$,
we have $S(D_\alpha, \alpha \zeta)=\alpha^p S(D,\zeta)$,
where $\alpha \zeta := \{\alpha x : x \in \zeta\}$.
%for any subset $U$ of $\R^d$, and
%for $\alpha >0$, $ U_\alpha \,=\,\{\alpha z, z \in U\}$. 
\end{itemize}
\label{def5}
\end{defi}
\begin{defi}
Two measurable functionals
$S_1$ and $S_2$ are said to be pointwise close of order $p\geq 0$ if
for every rectilinear cube $D \subset \R^d$
and every $\zeta \in \calA_{lf}^d$, we have
\begin{equation}
\frac{ |S_1(D,\zeta)-S_2(D,\zeta)|}{
\diam(D)^p}
\,=\,o
%\bigg( 
\left((\card{{\zeta \cap D}})^{(d-p)/d} \right) 
%\bigg),
\label{close0}
\end{equation}
in the sense that there is a function $g: \Z^+ \to \R^+ $
%with $\lim_{k \to \infty}g(k) =0$, such that
with $\lim_{k \to \infty} k^{(p-d)/d} g(k) =0$, such that
the left side of \eq{close0} is
bounded by 
%$g\left((\card{{\zeta \cap D}})^{(d-p)/d} \right)$
$g\left(\card{{\zeta \cap D}} \right)$
uniformly over $D$ and $\zeta$. 
\end{defi}

Regarding terminology, note that the condition
that  $S$ be localized 
(in the sense of Definition \ref{def5})
is stronger than the condition that $S$ be local 
(in the sense of Definition \ref{defi:sub}).

%\begin{remark}
A translation invariant measurable function $S$ that
is also homogeneous of order $p \geq 0$ and
satisfies $S(D,\emptyset)=0$ for all $D$, is
called a {\em Euclidean functional}. See 
\cite{Yukich-98} for examples of Euclidean functionals. 
%\end{remark}

We are now in position to state an umbrella theorem for
translation invariant functionals on the random parking measure.
Let $\xi$ and $\xi^D$ be as in Definition \ref{defi:RP}
\begin{theorem}\label{th:umbrella}
Let $\rho_0>0$. Let $1\leq p<d$, and let $S$ and $T$ be measurable
translation invariant smooth functionals of order $p$.
Assume that $S$ is localized and almost subadditive of order $p$,
and $T$ is superadditive.
If $S$ and $T$ are pointwise close of order $p$, then there
exists deterministic $\overline{S}\in \R^+$ such that for
all $D\in \calO(\R^d)$, we have almost surely 
\begin{equation*}
\lim_{R\to \infty}\frac{S(D_R,\xi^{D_R})}{|D_R|}\,=\,\lim_{R\to \infty}\frac{S(D_R,\xi)}{|D_R|}\,=\,\overline{S}.
\end{equation*}
%
%where $\xi^R$ and $\xi$ are the random parking measures of intensity 
%$\rho_0$ on $D_R$ and on $\R^d$, respectively.
%and $D_R$ denotes the dilation of $D$ by a factor $R$:
% $D_R=\{Rx \,,\,x\in D\}$.
\end{theorem}
\begin{remark}
We do not assume homogeneity of $S$ and $T$
in the statement of this theorem. However,
if in fact we do assume that
$S$ and $T$ are homogeneous of order 
$p$, then to check 
$S$ is smooth of order $p$ we need to check
\eq{smooth0} only
for $D = [0,1]^d$, and to check
$S$ and $T$ are pointwise close we 
need to check \eq{close0} only for $D =[0,1]^d$. 

\end{remark}
\begin{remark}
%Before we proceed with the proof, let us stress the fact 
By \cite[Lemmas~3.2,~3.5, and~3.7]{Yukich-98}, the $p$-power
weighted total length version of the traveling salesman problem, the minimal spanning tree, and the minimal matching all satisfy the assumptions
of Theorem~\ref{th:umbrella} with exponent $p$.
In particular, these examples are homogeneous and
the associated superadditive Euclidean functions $T$ are the
so-called canonical boundary functionals (denoted by $S_B$ in 
\cite{Yukich-98}).  

We believe there could be  interesting examples where 
homogeneity fails but our theorem  is still applicable,
for example to do with the number of percolation components.
See \cite{Yukich-98}, page 47 for a discussion
of examples where homogeneity fails.
\end{remark} 
\begin{remark}
The proof of Theorem~\ref{th:umbrella} shows that its
conclusion holds for any bounded connected subset
$D$ of $\R^d$ whose boundary has zero $d$-Lebesgue measure.
\end{remark}
\begin{remark}
In the proof of Theorem~\ref{th:umbrella} we show that $S$
has the averaging property on $\calA_{\rho_1,\infty}$ with
respect to $D$ for any $D\in \calO(\R^d)$ and
for any $\rho_1>0$ (in particular we need only the hard-core property).
\end{remark}

\begin{proof}[Proof of Theorem~\ref{th:umbrella}]
We split the proof into five steps.
We first prove that $S$ satisfies the insensitivity to boundary effects.
In the second step, we address the averaging property of $T$
on cubes with vertices in $\Z^d$, then that of $S$ on cubes
with vertices in $\R^d$,
and finally on any domain in $\calO(\R^d)$.
Appealing then to Theorem~\ref{th:sub} we will conclude the proof.

\medskip

\step{1}{Insensitivity of $S$ to boundary effects}
Let $D \in \calO(\R^d)$ and 
$0\leq \alpha <1$. Let $\hat Q$ be a cube aligned with the canonical basis of $\R^d$ and which contains $D$.
By localization and smoothness of $S$,
for all $\zeta\in \calA_{\rho_1,\infty}$,
\begin{multline*}
|S(D_{R},\zeta)-S(D_{R,R^\alpha},\zeta)|\,=\,|S(\hat Q_R,\zeta \cap D_R)-
S(\hat Q_R,\zeta\cap D_{R,R^\alpha})| \\\,\leq\, CR^p (\card{{\zeta}\cap D_R 
\setminus D_{R,R^\alpha}})^{(d-p)/d}.
\end{multline*}
Since the boundary of $D$ has zero $d$-Lebesgue
measure, we can cover the boundary of $D_R$ by
$o(R^{(1-\alpha)d})$ balls of radius $R^\alpha$.
Provided $R \geq 1$,
the corresponding balls of radius $3R^{\alpha}$
cover $D_R \setminus D_{R,R^\alpha}$ and by the hard-core
condition, each such ball contains  $O( R^{d \alpha})$
points of $\zeta$. Hence $\card{\zeta \cap D_R \setminus D_{R,R^\alpha}}$
is $o(R^{d})$, and hence
$ |S(D_{R},\zeta)-S(D_{R,R^\alpha},\zeta)| = o( R^d).  $
Thus (\ref{insensalpha}) holds,  so
$S$
is insensitive to
to boundary effects on $\calD(D)$. 
%for all $0\leq \alpha<1$.

\medskip

\step{2}{Superadditive ergodic theorem on rectangles with vertices in $\Z^d$}
Let $\zeta$ be a stationary ergodic random lattice
taking values in $\calA_{\rho_1,\infty}$.
Let $\calR$ denote the collection
of half-open  rectilinear rectangles in $\R^d$ with vertices in
$\Z^d$, and let 
$\calC$ be the class of all cubes in $\calR$.
By localization $S(D,\emptyset)$ is a constant $\tilde{S}(\emptyset)$
independent of $D$.
%Since $T(D,\emptyset)=0$ for all $D\in \calO(\R^d)$,
%we have 
By pointwise closeness,
there is a  constant $C$ such that
for every  $D \in \calC$ 
we have  that
$$
T(D,\emptyset) \leq S(D,\emptyset) + C ( \diam (D) )^p  
%
% \leq C' ( \diam (D))^p
$$
and hence  
by taking $\zeta_1=\emptyset$ and $\zeta_2=\zeta$ in the definition of 
the smoothness of $T$, 
%all $\xi\in \calA_{\rho_1,\infty}$, 
we have
for a possibly different $C$ that for all $D \in \calC$,
almost surely
\begin{equation*}
T(D,\zeta) \, \leq \, C\diam(D)^p(1+\card{\zeta\cap{D}})^{(d-p)/d}.
\end{equation*}
The hard-core condition combined with an elementary geometric argument
imply there exists $C>0$ depending only on $\rho_1$ such that
$\card{\zeta\cap{D}}\leq C|D|$ for every $D \in \calR$.
% rectilinear rectangle $D$ with vertices in $\Z^d$, 
Hence, there is a further $C$ such that if $D \in \calC$ then
\begin{equation}
\Ee T(D,\zeta) \, \leq \, C|D|.
\label{unifT}
\end{equation}
However, this implies that \eq{unifT} also holds
whenever $D \in \calR $,
%is a half-open rectilinear rectangle
%with vertices in $\Z^d$,
since if not, then
we could take a large cube in $\calC$ which was the union of
disjoint 
translates of $D$ and using
stationarity of $\zeta$ and translation-invariance
and superadditivity of $T$, we would
have a violation of \eq{unifT} for this cube.

We may now apply the Akcoglu-Krengel superadditive
ergodic theorem \cite{Akcoglu-Krengel-81} 
%(in the form of \cite[Theorem~?????]{Yukich-98}) 
to the functional $D\mapsto T(D,\zeta)$
which is stationary  under integer shifts.
%In particular, 
By the Akcoglu-Krengel theorem,
%We show next that 
for any unit cube 
$\hat Q \in \calC$ 
%which is contained in the positive orthant,
there is a random variable
$\overline{S}(\hat Q)$ such that
almost surely
\begin{equation}\label{eq:pr-propo-1}
\lim_{R\to +\infty,R\in \N} \frac{T(\hat Q_R ,\zeta)}{R^{d}} = \overline{S} (\hat Q).  
\end{equation}
%\Comment{MP}
Indeed, if $\hat Q$ is contained in the positive orthant
then  the sequence of cubes $\hat Q_R$ is regular in the
sense of Akcoglu and Krengel. If not, then (since
it is a unit cube) $\hat Q$ is contained
in some other orthant and the argument is similar.
%\begin{equation}\label{eq:pr-propo-1}
%\lim_{R\to +\infty,R\in \N}\frac{T([0,R]^d,\zeta)}{R^d}\,=\,\overline{S}.
%\end{equation}

By the pointwise closeness of
$S$ and $T$, and the hard-core condition on $\zeta$,
\eqref{eq:pr-propo-1} implies that 
for  $\hat Q \in \calC $ a unit cube we have
almost surely
\begin{equation}\label{eq:pr-propo-1-2}
\lim_{R\to +\infty,R\in \N}\frac{S(\hat Q_R,\zeta)}{R^d}\,=\,\overline{S}(\hat Q).
\end{equation}
We claim 
that $\overline{S}(\hat Q)$ is deterministic and the same for all
unit cubes $\hat Q \in \calC$ (hence subsequently denoted by simply $\overline{S}$).
Indeed, given such a $\hat Q$ and given $z\in \R^d$,
we may deduce
by localization and smoothness of $S$ that
\begin{equation*}
|S(\hat Q_R,\zeta)-S(\hat Q_{R} \cap ( \hat Q_R + z),\zeta)|
\,\leq \,C R^p(\card{\zeta \cap
\hat Q_R \setminus ( \hat Q_R + z)) }^{(d-p)/d}.
\end{equation*}
By the hard-core property of $\zeta$,  there is a constant
$C$ depending on $\rho_1$ and $z$ such that 
\linebreak
$\card{\zeta\cap {\hat Q_R\setminus (\hat Q_{R} +z)}}$ is asymptotically bounded
by $C R^{d-1}$, so that
\begin{equation}
\limsup_{R\to \infty, R \in \N}
\frac{|S(\hat Q_R,\zeta)-S(\hat Q_R \cap (\hat Q_R +z),\zeta)|}{R^d } \,= \,0
\label{0628a}
\end{equation}
and likewise
$R^{-d}|S(\hat Q_R +z,\zeta)-S(\hat Q_R \cap (\hat Q_R +z),\zeta)| \to 0$
as $R \to \infty$ through $\N$. Hence  by \eq{eq:pr-propo-1-2}
we have
\begin{equation*}
\lim_{R\to +\infty,R\in \N}\frac{S(\hat Q_R +z ,\zeta)}{R^d}\,=\,\overline{S}(\hat Q).
\end{equation*}
Hence $\overline{S}(\hat Q)$ is shift invariant,
and therefore constant by ergodicity. Also, by translation invariance
of $S$ and stationarity of $\zeta$, the distribution of
$\overline{S}(\hat Q)$ is the same for all unit cubes $\hat Q \in \calC$,
so we have justified our claim.

\medskip
\step{3}{Extension of \eqref{eq:pr-propo-1-2} to general cubes}
Let $\zeta$ be as in Step 2.
%{\bf Is  there a quicker argument for this than what follows?}
Let $\hat Q \in \calC$ be a cube of side $k$ with
$k \in \N$.
Let $\hat Q^i, 1 \leq i \leq k^d$, be disjoint unit cubes in $\calC$,
whose union is $\hat Q$.  Then by superadditivity of $T$,
\begin{eqnarray}
%\bar{S} =
\liminf_{R \to \infty, R \in \N} \frac{T(\hat Q_R,\zeta)}{R^{d}}  \geq
\liminf_{R \to \infty, R \in \N} \frac{1}{R^{d}}\sum_{i=1}^{k^d}  T((\hat Q^i)_R,\zeta)
= k^d\bar{S}.
\label{0629a}
\end{eqnarray}
Also, by repeated use of 
the almost subadditivity and localization of $S$,
as  in \cite[Remark 1 (3.5)]{Yukich-98},
there is a constant $C(k)$ such 
that
$$
S(\hat Q_R,\zeta)\,\leq\,
\sum_{i=1}^{k^{d}} S((\hat Q^{i})_R,\zeta)+C(k) R^p, 
$$
so that
\begin{eqnarray}
\limsup_{R \to \infty, R \in \N} \frac{S(\hat Q_R,\zeta) }{R^{d}} \leq
\limsup_{R \to \infty, R \in \N} \frac{1}{R^{d}}  \sum_{i=1}^{k^d}  S((\hat Q^i)_R,\zeta)
= k^d\bar{S}. 
\label{0629b}
\end{eqnarray}
Moreover by closeness of $T$ and $S$,
and the hard-core property of $\zeta$,
\begin{eqnarray}
\limsup_{R \to \infty, R \in \N}
\frac{1}{R^{d}} | T(\hat Q_R,\zeta) - S(\hat Q_R, \zeta)|   = 0.
\label{0629c}
\end{eqnarray}
Combining  \eq{0629a}, \eq{0629b} and \eq{0629c}
we may deduce that \eq{eq:pr-propo-1}
and \eq{eq:pr-propo-1-2}
hold for all $\hat Q \in \calC$, with $\overline{S}(\hat Q) = |\hat Q|\overline{S}$.

\medskip
%\step{3}{Extension of \eqref{eq:pr-propo-1} to $S^p$ for general cubes and
% proof that $\overline{S}^p$ is deterministic}
Next, for $\hat Q \in \calC$ we relax
the assumption that $R\in \N$ in the asymptotic formula 
\eqref{eq:pr-propo-1-2}.
Denoting by $[R]$ the integer part of $R$,
by using the localization and smoothness of $S$ 
and the hard-core property of $\zeta$ as in
the proof of \eq{0628a} above, we have
%
%\begin{equation*}
%|S(\hat Q_R,\zeta)-S(\hat Q_{[R]},\zeta)|\,\leq \,C R^p(\card{\zeta 
%\cap {\hat Q_R\setminus \hat Q_{[R]}}})^{(d-p)/d}.
%\end{equation*}
%
%By the hard-core property, 
%$\card{\zeta\cap {\hat Q_R\setminus \hat Q_{[R]}}}$ is asymptotically bounded
%by $\rho_1^{-d}R^{d-1}|\partial \hat Q|$, so that
%
\begin{equation*}
\limsup_{R\to \infty}\frac{|S(\hat Q_R,\zeta)-S(\hat Q_{[R]},\zeta)|}{R^d}\,= \,0.
\end{equation*}
Combined with \eqref{eq:pr-propo-1-2} this shows that
\begin{equation}\label{eq:pr-propo-2}
\lim_{R\to +\infty}\frac{S(\hat Q_R,\zeta)}{|\hat Q_R|}\,=\,\overline{S}.
\end{equation}
Now let $\hat Q$ be a general rectilinear cube
in $\R^d$.
% aligned with the canonical basis of $\R^d$.
% and of sidelength $c>0$.
For all $\tilde R>0$ large enough, there exist $k_{\tilde R}\in \N$ with
$k_{\tilde R} \geq \tilde R-2>0$ 
and a rectilinear cube $\hat Q^{k_{\tilde R}}$ with vertices in $\Z^d$ and sidelength $k_{\tilde R}$ such that 
$\hat Q^{k_{\tilde R}}\subset \hat Q_{\tilde R}$.
As above, the hard-core property of $\zeta$ yields the estimate 
\begin{equation*}
\card \zeta \cap \hat Q_{R\tilde R}\setminus (\hat Q^{k_{\tilde R}})_R
= 
\card \zeta \cap ( \hat Q_{\tilde R}\setminus \hat Q^{k_{\tilde R}} )_R
\,\leq\,C R^d {\tilde R}^{d-1}
\end{equation*}
so that localization and smoothness of $S$ yields 
\begin{equation*}
\limsup_{R\to \infty} \frac{|S((\hat Q^{k_{\tilde R}})_R,\zeta)-S(\hat Q_{R\tilde R},\zeta)|}{|\hat Q_{R\tilde R}|}\,\leq\,  C{\tilde R}^{(p-d)/d},
\end{equation*}
so using  \eqref{eq:pr-propo-2} applied to $\hat Q^{k_{\tilde R}}$,
and the arbitrariness of $\tilde R$, we obtain
\begin{equation}\label{eq:pr-propo-3}
\lim_{R\to +\infty}\frac{S(\hat Q_R,\zeta)}{|\hat Q_R|}\,=\,\overline{S}.
\end{equation}

\step{4}{Extension of \eqref{eq:pr-propo-3} to general domains $D\in \calO(\R^d)$}
Let $\zeta$ be as in Step 2.
Let $D^{+}$ be a cube aligned with the canonical basis of $\Z^d$, containing $D$, and such that  
its sidelength $\delta$ is bounded by $\diam(D)$.

We proceed by approximation.  For every integer
$j\geq 1$, let $\e_j= 2^{-j} \delta$, and
split $D^+$ into $2^{jd}$ cubes $\{\hat Q^{j,i}\}_{i\in \{1,\dots,2^{jd}\}}$ of 
sidelength $\e_j$.  Define $D^j=\cup_{i\in I^{j}}\hat Q^{j,i}$ as the union
of those cubes $\hat Q^{j,i}$ (indexed by $i\in I^{j}$) 
whose intersection with $D$ is not empty, and
define $\underline{D}^j =\cup_{i\in \underline{I}^{j}}\hat Q^{j,i}$
as the union of those cubes $\hat Q^{j,i}$
(indexed by $i\in \underline{I}^{j}$) 
which are contained in $D$.
In particular, $I^{j}$ is finite and $\underline{D}^j \subset
D\subset D^{j}\subset D^+$.

Let $\eta>0$. Using the fact 
that the boundary of $D$ has zero Lebesgue measure,
choose $j$ such that 
%
%\begin{equation*}
$
|D^{j}\setminus \underline{D}^j|\,\leq\,\eta,
$
%\end{equation*}
so that $D^{j}\setminus \underline{D}^j$
consists of at most $\eta \e_j^{-d}$ cubes of side $\e_j$.
This $j$ will be fixed for a while.
By the hard-core constraint there is a constant
$C$ depending only on $d$ and $\rho_1$, such that
for $R$ so large that $R \e_j >1$ we have
\begin{equation}
\card{\zeta \cap (D^{j}\setminus \underline{D}^j)_R} \leq
C
\eta \e_j^{-d}  
(R \e_j)^d 
= CR^{d}\eta.
\label{0629e}
\end{equation}
% is asymptotically bounded by $
%kC\rho_1^{-d}R^{d}\eta$,
%$$
% large enough As $R \to \infty$ with $j$ fixed, 
% $\card{\zeta \cap (D^{j})_R \setminus D_R}$ is asymptotically bounded by
% $C\rho_1^{-d}R^{d}\eta$,
%
Hence
by localization and smoothness of $S$,
\begin{eqnarray}
\limsup_{R\to \infty}\frac{|S(D_R,\zeta)-S((D^{j})_R,\zeta)|}{R^d}
= \limsup_{R\to \infty}\frac{|S(D^+_R,\zeta \cap D_R)-S(D^+_R,\zeta \cap
(D^j)_R)|}{R^d}
\nonumber \\
\leq\,C\eta^{(d-p)/d}.
\label{0629d}
\end{eqnarray}
Using localization and repeatedly using 
almost-subadditivity,
as  in \cite[Remark 1 (3.5)]{Yukich-98},
yields  for some constant $C(j)$ that
\begin{equation*}
S((D^{j})_R,\zeta)\,=\,S(D^{+}_R,\zeta\cap {(D^{j})_R})\,\leq\,\sum_{i=1}^{2^{jd}} S((\hat Q^{j,i})_R,\zeta\cap {(D^{j})_R})+C(j) R^p. 
%2^{(d-p)j},
\end{equation*}
%
%where we have first used that $S$ is localized.
Since $S((\hat Q^{j,i})_R,\zeta\cap{(D^{j})_R})$ is
a constant $\tilde{S}(\emptyset)$ for $i\notin I^{j}$, this turns into
\begin{equation*}
S((D^{j})_R,\zeta)\,\leq\,\sum_{i\in I^{j} }S((\hat Q^{j,i})_R,
\zeta)+C(j) R^p. 
%2^{(d-p)j}.
\end{equation*}
Hence, using \eq{0629d} and using
\eqref{eq:pr-propo-3} for each cube $\hat Q^{j,i}$, $i\in  I^{j}$, this yields
\begin{equation}
\label{eq:pr-propo-limsup}
\limsup_{R\to +\infty} 
%\frac{S(D_R,\zeta)}{R^d}
R^{-d}
S(D_R,\zeta)
\,\leq\,
C \eta^{\frac{d-p}{d}}
+ \sum_{i\in I^j}|\hat Q^{j,i}| \overline{S}\,\leq\,
C \eta^{\frac{d-p}{d}}
+
(|D|+\eta)\overline{S}.
\end{equation}

We now turn to the lower bound.
%For $j \in \N$,
%we let $\e_j=2^{-j}\delta$, and split $D^+$ into $2^{jd}$ cubes $\hat Q^{j,i}$ of sidelength $\e_j$.
%We then
% define $\underline{D}^{j}= D^+ \setminus \underline{D}^j$.
%\cup_{i\in \underline{I}^{j}}\hat Q^{j,i}$ as the union of those cubes $\hat Q^{j,i}$ 
%(indexed by $i\in \underline{I}^{j}$) 
%which are included in $D$.
%In particular, $\underline{I}^{j}$ is finite, and $D^+\subset \underline{D}^{j} \cup D$.
%Let $\eta>0$ and choose $j$ such that 
%
%\begin{equation*}
%|D|+\sum_{i\in \underline{I}^{j}}|\hat Q^{j,i}|-|D^+|\,\leq\,\eta.
%\end{equation*}
%
By almost subadditivity and localization of $S$,
\begin{equation}
S(D^+_R,\zeta)\,\leq\,S(D_R,\zeta)+S(D^+_R\setminus D_R,\zeta)+CR^p.
%S\bigg((\sqcup_{i\in I^\e}\hat Q^{\e,i}_R)\setminus D_R,\zeta\bigg).
\label{0629f}
\end{equation}
By localization and smoothness of $S$,
\begin{equation*}
|S(D^+_R\setminus D_R,\zeta)-S(D^+_R\setminus (\underline{D}^{j})_R,\zeta)| 
\,\leq\,CR^p(\card{\zeta\cap {D_R\setminus (\underline{D}^{j})_R}})^{(d-p)/d},
\end{equation*}
and by \eq{0629e} this gives us
%The hard-core condition implies that $\card{\zeta\cap
% {D_R\setminus (\underline{D}^{j})_R}}$ is asymptotically bounded by 
%$\rho^{-d} R^{d}\eta$, so that 
%the smoothness estimate turns into
%
\begin{equation}
\limsup_{R\to \infty}\frac{|S(D^+_R\setminus D_R,\zeta)-S(D^+_R\setminus (\underline{D}^{j})_R,\zeta)|}{R^d} \,\leq\,C\eta^{1 - p/ d}.
\label{0629g}
\end{equation}
We now treat $S(D^+_R\setminus (\underline{D}^{j})_R,\zeta)$: by
localization and almost-subadditivity,
\begin{eqnarray}
S(D^+_R\setminus (\underline{D}^{j})_R,\zeta)&=&S(D^+_R,\zeta\cap {D^+_R\setminus (\underline{D}^{j})_R}) \nonumber  \\
&\leq &\sum_{i=1}^{2^{jd}}S((\hat Q^{j,i})_R,\zeta\cap {D^+_R\setminus (\underline{D}^{j})_R})+C(j) R^p \nonumber \\
&=&\sum_{i\notin \underline{I}^{j}}S((\hat Q^{j,i})_R,\zeta)+C(j)
R^p.
\label{0629h}
\end{eqnarray}
Applying
\eq{0629f}, \eq{0629g} and \eq{0629h} in turn,
%Appealing again to
and using
\eqref{eq:pr-propo-3} for each cube $\hat Q^{j,i}$ and for $D^+$, we obtain
\begin{eqnarray}
\liminf_{R \to + \infty}
R^{-d} S(D_R,\zeta)
&
\geq
& |D^+| \overline{S} -
\limsup_{R \to + \infty}
R^{-d} S(D^+_R \setminus D_R,\zeta)
\nonumber
\\
& \geq & |D^+| \overline{S} - C \eta^{1-p/d} - 
\limsup_{R \to + \infty}
R^{-d} S(D^+_R \setminus (\underline{D}^j)_R,\zeta)
\nonumber
\\
& \geq & |D^+| \overline{S} - C \eta^{1-p/d} - 
|D^+ \setminus \underline{D}^j| \overline{S}
\nonumber
\\
& \geq & (|D|- \eta) \overline{S} - C \eta^{1-p/d}. 
%\label{eq:pr-propo-liminf}
\nonumber
\end{eqnarray}
%these last four estimates combine to
%
%\begin{equation*}
%\frac{|D^+|}{|D|} \overline{S}^p\,\leq\,\liminf_{R\to+\infty}\frac{S(D_R,\zeta)}{|D_R|}+\sum_{i\notin \underline{I}^{j}} \frac{|\hat Q^{j,i}|}{|D|} \overline{S}^p,
%\end{equation*}
%%
%from which we deduce 
%%
%\begin{equation*}
%\liminf_{R\to+\infty}\frac{S(D_R,\zeta)}{|D_R|} \,\geq\,\frac{|D|-\eta}{|D|}\overline{S}^p,
%\end{equation*}
%
%and therefore 
%
%\begin{equation}\label{eq:pr-propo-liminf}
%\liminf_{R\to+\infty}\frac{S(D_R,\zeta)}{|D_R|} \,\geq\,\overline{S}^p,
%\end{equation}
%
Combined with \eq{eq:pr-propo-limsup}
and using
the arbitrariness of $\eta$, this
concludes the proof of
the averaging property (\ref{avging}) for $S$.

\step{5}{Conclusion}
We can now
apply Theorem~\ref{th:sub} to $S$, since in Step~1 we have shown
the insensitivity to boundary effects, and
in Step~4 we have shown the averaging property.
% It remains to notice that the Euclidean functional $S$
% satisfies the uniform boundedness
%assumption by smoothness, using the hard-core condition
% for elements of $\calA_{\rho_1,\infty}$.
\end{proof}
%

%%%%%%%%%%%%%%%%%%%%%%%%%%%%%%%%%%%%%%%%%%%%%%%%%%%%%%%%%%%%%%%%%%%%%%%%%%
%%%%%%%%%%%%%%%%%%%%%%%%%%%%%%%%%%%%%%%%%%%%%%%%%%%%%%%%%%%%%%%%%%%%%%%%%%

\section{Application to the derivation of rubber elasticity}\label{sec:rubber-thermo}

\subsection{A discrete model for rubber and its thermodynamic limit}\label{sec:rubber}

Let $0 < \rho_1 < \rho_2 < \infty$.
Suppose that $\calL \in \calA_{\rho_1,\infty}$
and $\calL$ has at least $d+1$ elements (so its convex hull
has strictly
positive $d$-Lebesgue measure,
since $\calL$ is assumed general).
%. Since
%If $\calL$ is general and its convex hull has strictly
%positive $d$-Lebesgue measure,
%then 
Then there is a unique Delaunay triangulation
of the convex hull of $\calL$, by simplices with edges given
by the edges of the Delaunay graph of $\calL$, which is itself the dual graph 
of the Voronoi tessellation (see \cite{Okabe-00}, and \cite {Delaunay-34} for Delaunay tessellations of $\R^d$, and for instance \cite {Fortune-95} for Delaunay
tessellations of a bounded domain).
By convention we consider simplices to be open sets.
% 
%If $\calL$ is not general, there exists a  possibly non-unique
% Delaunay triangulation 
%of the convex hull of $\calL$.
Let $\calT(\calL)$ denote 
the 
%a
Delaunay triangulation of the convex hull of $\calL$.
For any $d+1$ points of $\calL$, the
simplex generated by these points is in $\calT(\calL)$
if and only if no point of $\calL$ lies in the interior
of the circumsphere of these points.
% 
%$\calT(\calL)$ is such that the circumsphere associated 
%with the $d+1$ vertices of any simplex $T\in \calT(\calL)$ 
%contains no point of $\calL$ in its interior. Conversely,
%% if $\calL$ is general, 
%if the interior of the circumsphere associated with
%$d+1$ points of $\calL$ contains no point 
%of $\calL$ then the simplex generated by those $d+1$ points belongs to
% $\calT(\calL)$.
We denote by $\calN(\calL)$ the associated  neighbour pairs,
that is, those unordered pairs of points $\{x,y\}$
such that $(x,y)$ is an edge of $\mathcal{T} := \mathcal{T}(\calL)$.
For all $\e>0$ and all open sets $D\in \calO(\R^d)$, 
we define 
%this allows us to uniquely define 
a space
of continuous piecewise-affine functions ${\mathfrak{S}}_\e^D (\mathcal L)$ on $D$, by
\begin{equation}\label{eq:affine}
\mathfrak{S}_{\e}^D(\mathcal L):=\{ u\in C(D,\R^n)\,\big|\, \forall T\in
\mathcal{T}(\mathcal L),\, \mbox{with}\, \e T\cap D\neq\emptyset,\,
u_{|\e T\cap D}\, \mbox{is affine}\}.
\end{equation}
Let $\overline{D}$ denote the closure of $D$. 
From now on, we identify $u:\e{\mathcal
L}\cap \overline{D}\to\R^n$ with its class of piecewise-affine interpolations
(still denoted by $u$) in ${\mathfrak{S}}_\e^D (\mathcal L)\subset
W^{1,\infty}(D,\R^n)$. 
Note that the extension of $u: \e{\mathcal
L}\cap \overline{D}\to\R^n$ to $D\setminus \cup_{T\in \calT,T\subset D}T$ is not uniquely defined --- as we shall see, the energy
under consideration does not depend on the extension.
In order to define an energy functional on the set  ${\mathfrak{S}}_\e^D (\mathcal L)$, we first introduce
the following energy functions:
\begin{defi}\label{hypo:ener-SRL}
Let $p>1$.
We denote by ${\mathcal U}_p$ the subset of functions $f_{nn}$ of
% ${\mathcal C}^0(\R^d\times\R^n,\R^+)$ 
$C(\R^d\times\R^n,\R^+)$ 
for which there exists 
$C>0$ 
such that, for all $z\in \R^d$ and $s \in \mathbb{R}^n$,
\begin{eqnarray}
\frac{1}{C}|s|^p-C \leq f_{nn}(z,s) \leq C(|s|^p+1).  \label{eq:gcSRL}
\end{eqnarray}
We denote by ${\mathcal V}_p$ the subset of functions $W_\vol$ of
$C(\M^{n\times d},\R^+)$ for which there exists $C>0$ 
such that for all $\Lambda\in\M^{n\times d}$,
\begin{eqnarray}\label{eq:sgc-vol-SRL}
W_{\vol}(\Lambda) \leq C(|\Lambda|^p+1).
\end{eqnarray}
\end{defi}
Let $p>1$ and $f_{nn}\in {\mathcal U}_p,W_\vol \in {\mathcal V}_p$.
For all $u\in L^1(D,\R^n)$ and open $D \in \calO(\R^d)$ we then set
\begin{equation}\label{eq:def-ener-vol2}
F_\e^D(\calL,u):=
\left\{
\begin{array}{rl}
F_{nn,\e}^D(\calL,u)+F_{\vol,\e}^D(\calL,u) &\mbox{ if }u\in  {\mathfrak{S}}_\e^D (\mathcal L),
\\
+\infty&\mbox{ else},
\end{array}
\right.
\end{equation}
where 
%(with $\overline{D}$ denoting the closure of $D$) 
we set
% %
% \begin{eqnarray}\label{def:ener-nn-lr}
% F_{nn,\e}^D(\calL,u)&=& \dps{ \sum\limits_{\begin{array}{cc}
% (x,y)\in \mathcal{N}(\calL)\cr
% [\e x,\e y]\subset \overline{D} \end{array}} \e^d f_{nn}\left(y-x,\frac{u(\e y)-u(\e x)}{\e |y-x|} \right)  }, 
% \end{eqnarray}
% %
% and
% %
% \begin{eqnarray}\label{eq:def-volume}
% F_{\vol,\e}^D(\calL,u)&=& \dps{ \sum\limits_{\begin{array}{cc}
% T\in \mathcal{T}(\calL) \cr
% T\subset D_{\e^{-1}}\end{array}} \e^d |T|W_{\vol}(\nabla u_{|\e T}).   }
% \end{eqnarray}
% Alternatively (and equivalently), we can write
%{\bf [I find (\ref{altnn}) and (\ref{altvol})
%more transparent than 
%(\ref{def:ener-nn-lr}) and
%(\ref{eq:def-ener-vol2}) for understanding
%the limiting procedure. Perhaps just use these
%instead? MP]}
%(for some function $\tilde{f}_{nn}$ which is also in $\calU_p$)
\begin{eqnarray}\label{def:ener-nn-lr}%\label{altnn}
\quad
F_{nn,\e}^D(\calL,u)&=& \dps{ \sum\limits_{\begin{array}{cc}
\{x,y\}\in \mathcal{N}(\e \calL):
(x,y)\subset \overline{D} \end{array}} \e^d f_{nn}\left(\e^{-1}(y-x),
\frac{u(y)-u(x)}{|y-x|} \right)  }, 
\end{eqnarray}
and
\begin{eqnarray}\label{eq:def-volume}%\label{altvol}
F_{\vol,\e}^D(\calL,u)&=& \dps{ \sum\limits_{\begin{array}{cc}
T\in \mathcal{T}(\e \calL) :
T\subset D \end{array}}  |T|W_{\vol}(\nabla u_{|T}).   }
\end{eqnarray}

As announced, if $u_1,u_2\in  {\mathfrak{S}}_\e^D (\mathcal L)$ are such that 
$u_1=u_2$ on $\cup_{T\in \calT(\e \calL),T\subset D}T$, then $F^D_\e(\calL,u_1)=F^D_\e(\calL,u_2)$.

A suitable notion to study the convergence of the functional 
$F_\e^D(\calL):= F_\e^D(\calL,\cdot)$
is $\Gamma$-convergence.
We briefly recall its definition for the unfamiliar reader.
We say that a functional $F_\e:L^p(D,\R^n)\to [0,+\infty]$ $\Gamma$-converges to some functional $F:L^p(D,\R^n)\to [0,+\infty]$ if the
following two statements hold:
\begin{itemize}
\item[(i)] (liminf inequality) For all $v\in L^p(D,\R^n)$ and every sequence $v_k\in L^p(D,\R^n)$ which converges to $v$ in $L^p(D,\R^n)$
and every sequence $\e_k \to 0$, we have
\begin{equation*}
F(v)\,\leq\, \liminf_{k\to \infty}F_{\e_k}(v_k);
\end{equation*}
\item[(ii)]
% ($\Gamma$-limsup inequality)
(recovery sequence) For all $v\in L^p(D,\R^n)$ and every sequence $\e_k\to 0$, there exists a sequence $v_k\in L^p(D,\R^n)$ which converges to $v$ in $L^p(D,\R^n)$
and such that %\Comment{MP} [changed next display]
\begin{equation*}
F(v)\,= \, \lim_{k\to \infty}F_{\e_k}(v_k).
\end{equation*}
\end{itemize}
The notion of $\Gamma$-convergence is natural for minimization problems since it ensures
the convergence of minima and minimizers (recall that a set is
precompact if its closure is compact):
\begin{lemma}\cite[Section~7.1]{Braides-98}\label{lem:conv-min}
Suppose that $F_{\e_k}:L^p(D,\R^n)\to [0,+\infty]$ $\Gamma$-converges to some functional $F:L^p(D,\R^n)\to [0,+\infty]$ as $k\to \infty$.
If there exists a precompact sequence of minimizers $u_k$ of $F_{\e_k}$ in $L^p(D,\R^n)$, then
\begin{equation*}
\lim_{k\to \infty}\inf_{L^p(D,\R^n)}F_{\e_k}\,=\,\inf_{L^p(D,\R^n)}F.
\end{equation*}
Moreover, 
if $u_k \to u$ for some $u \in L^p(D,\R^n)$, then
$u$
% then every limit of a subsequence of $v_n$ 
is a minimum point for $F$.
\end{lemma}
For an introduction to $\Gamma$-convergence and its application to homogenization of integral functionals we refer the reader to \cite{Braides-02,Braides-98},
whence come the following useful definitions and
properties of integral functionals:
\begin{defi}
\label{quasidef}
We say that a function $W:\M^{n\times d}\to \R$ is \emph{quasiconvex} if for all $\Lambda \in \M^{n\times d}$,
\begin{equation*}
W(\Lambda) \,=\,\inf_u \left\{\int_Q W(\Lambda+\nabla u(x))dx\,:\,u\in W^{1,\infty}_0(D,\R^n)\right\}.
\end{equation*}
We say that it is {\em isotropic} if
$W (\Lambda \mathbf{R})=W (\Lambda) $ for all
$\Lambda \in\M^{n \times d},\mathbf{R}\in SO_d$.
We say that it satisfies a \emph{standard growth condition} of order $1<p<\infty$ if there exists a positive constant $C$ such that for all $\Lambda \in \M^{n\times d}$
\begin{equation}\label{eq:def-sgc}
C^{-1}|\Lambda|^p-C \,\leq\,W(\Lambda) \,\leq \, C(1+|\Lambda|^p).
\end{equation}

Let $D\in \calO(\R^d)$ be an open set and let $W:D\times \M^{n\times d}\to \R$ be a measurable function. We say that $W$ is {\em Carath\'eodory} if for almost every $x\in D$,
$\Lambda \mapsto W(x,\Lambda)$ is continuous.
In particular, this implies that for all $u\in W^{1,1}(D)$, the function $x\mapsto W(x,\nabla u(x))$ is measurable on $D$.

If in addition there exist $1<p<\infty$ and $C>0$ such that for almost every $x\in D$, $\Lambda \mapsto W(x,\Lambda)$ is quasiconvex
and satisfies the standard growth condition \eqref{eq:def-sgc}, then the integral functional $F:W^{1,p}(D)\to \R,  u\mapsto F(u)=\int_DW(x,\nabla u(x))dx$
is finite and lower-semicontinuous for the weak topology of  $W^{1,p}(D)$.

For $\Lambda \in \M^{n\times d}$,
define the function 
$\varphi_\Lambda:\R^d\to \R^n$ by $x \mapsto \Lambda x$.
\end{defi}

Recall that $Q_R :=(-R,R)^d$ for all $R>0$.
In \cite{Alicandro-Cicalese-Gloria-07b}, Alicandro, Cicalese and the first author proved the following $\Gamma$-convergence (or discrete homogenization) result:
\begin{theorem}\cite[Theorem~5]{Alicandro-Cicalese-Gloria-07b}\label{th:main-gen}
Let $0 < \rho_1 < \rho_2 < \infty$.
Let $\mathcal{L}$ be a stationary and ergodic stochastic lattice in $\calA_{\rho_1,\rho_2}$.
%which is almost  surely general.
Let $p>1$ and let $f_{nn}$  and $W_\vol$ be of 
class $\mathcal{U}_p$ and $\mathcal{V}_p$, respectively.
Let $D\in \calO(\R^d)$ be an open set, and let 
$F_\e^D(\calL)$ be the energy functional given by
(\ref{eq:def-ener-vol2}). 
Then  the functionals $F_\e^D(\calL)$ almost surely
$\Gamma$-converge as $\e\to 0$ to the deterministic integral
functional $F_{\ho}^D:L^{p}(D,\R^n)\to[0,+\infty]$ defined by
\begin{equation}
F_{\ho}^D(u)=\begin{cases}\int_D W_{\ho}(\nabla
u(x))\ud x & \text{if } u\in W^{1,p}(D,\R^n), \cr +\infty
&\text{otherwise},
\end{cases}
\end{equation}
where $W_{\ho}:\M^{n\times d}\to \R^+$ is a quasiconvex function which depends only on $f_{nn}$, $W_\vol$, and on the stochastic lattice, and which
satisfies a standard growth condition \eqref{eq:def-sgc} of order $p$.
Also for all $\Lambda \in \M^{n \times d}$, $W_{\ho}$ 
%In addition it 
satisfies the following asymptotic homogenization formula almost surely: 
\begin{multline}\label{eq:asymp-fo-de}
W_{\ho}(\Lambda)=\lim_{R \to \infty} \frac{1}{|Q_R|}\inf_{u} \bigg\{
F_1^{Q_R}(\calL,u)\ \bigg| \ u\in {\mathfrak{S}}_1^{Q_R}(\calL), ~
% \mbox{ such that } 
%u(x)=\Lambda\cdot x \\ \mbox{ if }x\in \calL \cap Q_R 
%\mbox{ and }  d_\infty(x,\partial Q_R) \leq 2\rho_2
u \equiv \varphi_\Lambda ~{\rm on}~ \calL \cap Q_R \setminus Q_{R- 2 \rho_2 }
\bigg\}.
\end{multline}
\end{theorem}
For the link between the above result and the derivation of rubber elasticity from the statistical physics of interacting polymer-chains, we refer
the reader to \cite[Section~4.1]{Alicandro-Cicalese-Gloria-07b}, \cite{Gloria-LeTallec-Vidrascu-08b}, and the references therein.
Note that the combination with Lemma~\ref{lem:conv-min} yields the convergence of minimum problems.

\medskip

Note that in \eqref{eq:asymp-fo-de}, the lattice $\calL$ and the domain $Q_R$ correspond to $\xi$ and $D_R$ in Theorem~\ref{th:sub}, respectively.
Yet, as we shall see, the function $(\calL,D_R)\,\mapsto \,S(\calL,D_R)$ defined by
\begin{multline*}
S(\calL,D_R) := \inf_{u} \bigg\{
F_1^{D_R}(\calL,u)\ \bigg| \ u\in {\mathfrak{S}}_1^{D_R}(\calL), ~
% \mbox{ such that } u(x)=\Lambda\cdot x \\ \mbox{ if }x
u \equiv \varphi_\Lambda \mbox{ on } \calL \cap D_R \setminus D_{R- 2 \rho_2} 
%\in \calL \cap D_R \mbox{ and }  d_\infty(x,\partial D_R) \leq 2\rho_2
\bigg\},
\end{multline*}
is not local in the sense of Definition~\ref{defi:sub}.

\subsection{Isotropic homogenized energy density and approximation result}\label{sec:ass-SET}

Let $0 < \rho_1 < \rho_0 < \rho_2 < \infty$, and
let $\xi$ denote the random parking measure of parameter $\rho_0$ in $\R^d$.
Also, let $\xi^R : = \xi^{Q_R}$ be the random parking measure of
parameter $\rho_0$ in the bounded domain $Q_R$.
Let $p>1$ and let $f_{nn}$  and $W_\vol$ be of 
class $\mathcal{U}_p$ and $\mathcal{V}_p$, respectively.

Proposition~\ref{prop:iso-ergo} answers the first question of the introduction: there does exist an ergodic admissible stochastic lattice which
is statistically isotropic, namely $\xi$:
% the random parking measure $\xi$ in $\R^d$. 
%We have:
%
\begin{theorem}\label{th:isotropy}
Suppose $n=d$.  
Suppose there exists
$\tilde f_{nn}:\R^+\times \R^d\to \R^+$ such that 
\begin{equation}\label{eq:hypo-isot}
\begin{array}{ll}
f_{nn}(z_1,z_2)=\tilde {f}_{nn}(|z_1|,z_2) & \forall \, z_1,z_2\in\R^d,
%\\
%W_\vol (\Lambda \mathbf{R})=W_\vol 
%(\Lambda) &\forall \, \Lambda \in\M^d,\mathbf{R}\in SO_d.
\end{array}
\end{equation}
and suppose also that $W_\vol$ is isotropic.
Then the energy density $W_{\ho}$ defined by the asymptotic formula 
\eqref{eq:asymp-fo-de} (with $\calL = \xi$ and
$F_\e^{Q_R}$  given by \eqref{eq:def-ener-vol2})
is well-defined and isotropic.
\end{theorem}
\begin{proof}[Proof]
Putting $\calL =\xi$ gives $\calL$ satisfying the conditions
of Theorem~\ref{th:main-gen}, so
the existence of $W_\ho$ is ensured by
that result.
The isotropy of $W_\ho$ is a direct consequence of the isotropy of $\xi$ in Proposition~\ref{prop:iso-ergo} and 
of \cite[Theorem~9]{Alicandro-Cicalese-Gloria-07b}.
\end{proof}
We now turn to a first version of the second question: in the case of the random parking measure, does the asymptotic formula \eqref{eq:asymp-fo-de}
hold if $\calL$ is replaced by $\xi^R$?
%=\xi^{Q_R}$ (that is, the random parking measure $\xi$ is replaced
% by its approximation on the bounded domain $Q_R$) ?
This question is particularly relevant for numerical approximations (see \cite{Gloria-LeTallec-Vidrascu-08b}), since
the approximation of $W_\ho$ on a bounded domain $Q_R$ also requires
an approximation of the random measure $\xi$ on the domain $Q_R$.
The answer is positive. 

Before we state the result proper, let us provide more details on the approximation procedure.
In Theorem~\ref{th:main-gen}, the infimum in \eqref{eq:asymp-fo-de} 
%is not local in the sense
%of Definition~\ref{defi:sub} since it 
uses the Delaunay tessellation associated with the point process $\calL$,
which is slightly nonlocal (the restriction of the Delaunay tessellation to
$Q_R$ depends on points of $\calL$
that lie outside $Q_R$, although in some bounded annulus around $Q_R$ due to the non-empty space condition).
Hence, for the approximation process we need to define a local version of \eqref{eq:asymp-fo-de} 
through the introduction of some tessellation depending only on $\calL$ in $Q_R$.
A natural choice would be to consider the Delaunay tessellation associated with $\calL\cap{Q_R}$. This is however not suitable
for the following reason: the edge lengths of this Delaunay tessellation are a priori only bounded by $R$ --- the Delaunay
tessellation of a point set is a tessellation of the convex hull of the point set into $d$-simplices, so that
there exist with positive probability configurations which have long edges. 
Arbitrarily large edge lengths are incompatible with the modeling of polymer chains (see  \cite[Section~4.1]{Alicandro-Cicalese-Gloria-07b} and \cite{Gloria-LeTallec-Vidrascu-08b}: in the case of unbounded edges, the associated  functions $f_{nn}$ are not of class $\mathcal{U}_p$, see Definition~\ref{hypo:ener-SRL}).
We deal with this issue by considering our energy functional
on a slightly smaller
region than the cube $Q_R$. 
%Let $\gamma=2d$.

\medskip
We now introduce an approximation of $W_\ho$ on $Q_R$ using $\xi^R$.
%To this aim, as in \eqref{eq:affine} and \eqref{eq:def-ener-vol2}, we define a functional space and an energy functional. 
The main result of this subsection is the following
approximation result:
\begin{theorem}\label{th:approx1}
%Let $0 < \rho_1 < \rho_0 < \rho_2 $.
%Let $\xi$ be the random parking measure in $\R^d$ 
%with parameter $\rho_0$, and for all $R>0$ let 
%$\xi^R$ denote the random parking measure in $Q_R$.
%Under the assumptions of Theorem~\ref{th:main-gen},
For all $\Lambda\in \M^{n\times d}$, 
the energy density $W_\ho$ defined in \eqref{eq:asymp-fo-de} 
(with $\mathcal{L}=\xi$) satisfies  almost surely the identity
\begin{equation*}
W_{\ho}(\Lambda)=\lim_{R \to \infty} 
\frac{1}{|Q_R|}\inf_u \left\{
F^{Q_{R- 2 \rho_2}}_1( \xi^R,u)
\ \bigg| \ u\in {\mathfrak S}^{Q_R}_1(\xi^R),u
\equiv \Lambda \mbox{ on } 
Q_{R- 2 \rho_2} \setminus Q_{R- 4 \rho_2}
%\partial Q_R
\right\}.
\end{equation*}
\end{theorem}
As we shall see,
Theorem~\ref{th:approx1} follows from Theorem~\ref{th:sub} applied to 
(some function related to) the infimum in \eqref{eq:asymp-fo-de}, seen
as a function of random measures and sets.
% %
% \begin{remark}
% Theorem~\ref{th:approx1} is of practical interest not only because the approximation of $W_\ho$ on $Q_R$
% only uses the random parking measure in $Q_R$ but also because the points on which Dirichlet boundary
% conditions are imposed are easily identified --- and independent of the realization of the random parking measure.
% ({\bf Should be able to extend to boundary conditions
% on just $\partial Q_{R- 2 \rho_2}$ if that is thought important}).
% \end{remark}
%
%
\begin{remark}
Other approximations of $W_\ho$ using the random parking measure on a bounded domain can be used instead of the one presented in Theorem~\ref{th:approx1}.
Let us mention two alternatives.

First, instead of restricting the energy to the domain $Q_{R-2\rho_2}$, one may consider the energy on the whole domain $Q_R$
provided we modify the random parking measure on $Q_R\setminus Q_{R-2\rho_2}$, typically by taking $\tilde \xi^R:=(\xi^R \cap Q_{R-2\rho_2})\cup \zeta^R$
where $\zeta^R$ is a suitable point set in $\partial Q_R$ (say deterministic). The point set $\zeta^R$ can be chosen so that the convex envelope of $\tilde \xi^R$
is $Q_R$, and such that any associated Delaunay triangulation of $Q_R$ has edge lengths bounded by $4\rho_2$. 
A result similar to Theorem~\ref{th:approx1} can be proved in this case, when the boundary conditions are imposed on $\zeta^R$ only (which greatly simplifies the
practical implementation, see \cite{Gloria-LeTallec-Vidrascu-08b}).

The second alternative is an adaptation of the popular ``periodization method'' in homogenization.
Recall that $\calP$ denotes a space-time Poisson process.
For all $R>0$ we define the $Q_R$-periodization of $\calP$ as
$\calP^R_\#:=\cup_{k\in \Z^d}(kR+\calP\cap (Q_R \times \R^+))$,
and let $\xi^R_\#$ be the output of the graphical construction of Subsection~\ref{sec:parking} associated with $\calP^R_\#$ in place of $\calP$.
We consider a $Q_R$-periodic tessellation $\calT^R_\#$ of $\R^d$ associated with $\xi^R_\#$.
This allows us to define a set of $Q_R$-periodic functions on $\R^d$:
\begin{equation*}
\mathfrak{S}_{\#}^R(\xi^R_\#) :=\{ u\in C(\R^d,\R^n)\,\big|\, u \mbox{ is }Q_R\mbox{-periodic}, \forall T\in
\calT(\xi^R_\#),\, u_{|T}\, \mbox{is affine}\}.
\end{equation*}
We then let $C^R_\#$ be a minimal connected periodic cell obtained as the union
of simplices of $\calT(\xi^R_\#)$.
Although we have not checked all the details, one should have that
\begin{multline*}
W_{\ho}(\Lambda)
\,=\,\lim_{R \to \infty} \frac{1}{|Q_R|}\inf_{u} \bigg\{
F_1^{C^R_\#}(\xi^R_\#,u)\ \bigg| \ u\in {\mathfrak{S}}_{\#}^{R}(\xi^R_\#) \\ \mbox{ such that } \frac{1}{|C^R_\#|}\int_{C^R_\#} \nabla u(x)dx=\Lambda \bigg\}
\end{multline*}
almost surely. 
In view of the results of \cite{Gloria-Neukamm-Otto-12} on the periodization method on a simplified model (discrete linear elliptic equation on $\Z^d$
with independent and identically distributed conductances), this approach (although it is much more delicate to implement in practice) should yield better convergence rates than the first two.
\end{remark}

The proof of
Theorem~\ref{th:approx1} requires a series of
lemmas.
The first of these provides elementary control 
of energy functionals for homogeneous deformations:
\begin{lemma}\label{lem:unif-bd}
%Let $0 < \rho_1 < \rho_2 < \infty$.
%Let $\zeta$ be a $(\rho_1,\rho_2)$-admissible point set in $\R^d$, 
%and let $\calT(\zeta)$ be an associated
%Delaunay tessellation of $\R^d$.
%Let $p>1$ and let $f_{nn}$  and $W_\vol$ be of 
%class $\mathcal{U}_p$ and $\mathcal{V}_p$, respectively.
%For $D\in \calO(\R^d)$, $\e>0$, let
% $F_\e^D$  be given by \eqref{eq:def-ener-vol2}.
% with $\calL=\zeta$, and $\calT=\calT(\zeta)$. 
There exists $C>0$ (depending only on $d, n, \rho_1,\rho_2$, $p$,
$f_{nn}$ and $W_\vol$),
such that for all $\zeta \in \calA_{\rho_1,\rho_2}$,
$D\in \calO(\R^d)$, all $\e>0$ and all $\Lambda \in \M^{n\times d}$,
$F_\e^D$  (given by \eqref{eq:def-ener-vol2}) satisfies
\begin{equation*}
0\,\leq\, F_\e^D(\zeta,\varphi_\Lambda)\,\leq\,C|D|(1+|\Lambda|^p).
\end{equation*}
%
%where $\varphi_\Lambda:\R^d\to \R^n,x\mapsto \Lambda x$.
\end{lemma}
\begin{proof}
For $\zeta \in \calA_{\rho_1,\rho_2}$,
%Since $\zeta$ is $(\rho_1,\rho_2)$-admissible, 
no edge  of $\calT(\zeta)$ has length greater
than $2 \rho_2$.
Hence the number of neighbours of vertices of $\calT(\zeta)$ is bounded by some 
finite constant depending only on $\rho_1$, $\rho_2$ and $d$.
This implies that the number of edges of the Delaunay tessellation 
$ \calT(\e \zeta)\cap D$
is bounded by
a constant times the volume $\e^{-d}|D|$. 
Using Definition~\ref{hypo:ener-SRL} and \eqref{eq:def-ener-vol2}, 
this yields the desired result.
\end{proof}

%
%Theorem~\ref{th:approx1}
% follows from Theorem~\ref{th:sub} applied to 
%Yet, to apply Lemma~\ref{lem:conv-min}, one needs to 

Our next lemma
shows that sequences of minimizers are precompact in $L^p(A,\R^n)$.
This precompactness is a consequence of an argument \`a la Fr\'echet-Kolmogorov.
We give this argument for completeness in the following
lemma since it does not explicitly appear in \cite{Alicandro-Cicalese-Gloria-07b}.
\begin{lemma}\label{lem:Frechet-Kolmogorov}
%Let $1<p<\infty$, 
Let $D\in \calO(\R^d)$ be an open set, let $(u_k)_{k \geq 1}$
be a bounded sequence in $L^p(D,\R^n)$, and let $A\in \calO(D)$ be an open set whose closure is contained in $D$.
If there exist an increasing continuous function $f:\R^+\to \R^+$ 
with $f(0)=0$ and a sequence $\e_k$ converging to zero such that for
all $h\in \R^d$ with $|h|$ small enough, and for all $k$,
\begin{equation}\label{eq:translation}
\|\tau_h u_k-u_k\|_{L^p(A,\R^n)} \,\leq \, f(|h|)+\e_k,
\end{equation}
where $\tau_hu_k:x\mapsto u_k(x+h)$ (whenever it is defined), 
then there exists $u\in L^p(D,\R^n)$ such that for a subsequence $u_k\to u$ in $L^p(A,\R^n)$.
\end{lemma}
\begin{proof}[Proof]
Since the sequence $u_k$ is bounded in $L^p(D,\R^n)$, by the Banach-Alaoglu theorem, there exists $u \in L^p(D,\R^n)$ such that $u_k$ converges weakly to $u$ for a subsequence (not relabeled).
By the lower-semicontinuity of the norm for the weak convergence, the weak 
convergence of $\tau_hu_k-u_k$ to $\tau_hu-u$ in $L^p(A,\R^n)$ for all $h$ small enough, and by \eqref{eq:translation}
\begin{equation}\label{eq:difference-quotient}
\|\tau_h u-u\|_{L^p(A,\R^n)} \,\leq \, \liminf_{k\to \infty} \|\tau_h u_k-u_k\|_{L^p(A,\R^n)} \,\leq\, f(|h|).
\end{equation}
Let $\rho_j$ be a sequence of smoothing kernels of unit mass with support in the ball $B(0,1/j)$ centred at zero and of radius $1/j$.
For $j$ large enough and for all $k$, the convolution $\rho_j\star u_k$ of $u_k$ with $\rho_j$ is in $L^p(A,\R^n)$.
In addition, 
by Jensen's inequality with measure $\rho_j(y)dy$, and
Fubini's theorem,
% and \eqref{eq:translation},
%
\begin{eqnarray*}
\int_A |\rho_j\star u_k-u_k|^p &\leq & \int_A \left(\int_{B(0,1/j)} |u_k(x-y)-u_k(x)|\rho_j(y)dy\right)^pdx\\
&\leq&
% \int_A \int_{B(0,1/j)} |u_k(x-y)-u_k(x)|^p\rho_j(y)dydx \\
%&= &
\int_{B(0,1/j)} \int_A |u_k(x-y)-u_k(x)|^pdx\rho_j(y)dy 
%\\
%&\leq & (f(j^{-1})+\e_k)^p.
\end{eqnarray*}
so that by \eqref{eq:translation},
\begin{equation}\label{eq:triang_1}
\|\rho_j\star u_k-u_k\|_{L^p(A,\R^n)} \leq f(j^{-1})+\e_k.
\end{equation}
%
%Indeed,
%
Likewise, by \eqref{eq:difference-quotient}
\begin{equation}\label{eq:triang_3}
\|\rho_j\star u-u\|_{L^p(A,\R^n)} \, \leq \,f(j^{-1}).
\end{equation}

By construction, $\{\rho_j\star u_k\}_k$ is an equi-continuous sequence of bounded functions in $C(A)$.
Indeed the functions are bounded since 
$$\|\rho_j \star u_k\|_{L^\infty(A)}\leq \|\rho_j\|_{L^\infty}\|u_k\|_{L^1(D,\R^n)}.$$
Also for all $x_1,x_2\in A$,
\begin{equation*}
|\rho_j \star u_k(x_1)-\rho_j \star u_k(x_2)|\,\leq\, |x_1-x_2|\|\rho_j\|_{\Lip}\|u_k\|_{L^1(D,\R^n)},
\end{equation*}
so that the equi-continuity follows by the uniform boundedness of $u_k$ in $L^1(D,\R^n)$.
Hence, by Ascoli's theorem, for all $j$ there exists $\tilde u_j \in C(A)$ such that $\rho_j\star u_k$ converges up to extraction 
to $\tilde u_j$ uniformly in $A$ as $k\to \infty$, and therefore in $L^p(A,\R^n)$.
Noting that for all $x\in A$,
\begin{equation*}
\lim_{k\to \infty}\rho_j \star u_k(x)\,=\,\lim_{k\to \infty} \int_{D} u_k(y)\rho_j(x-y)dy \,=\,\int_{D} u(y)\rho_j(x-y)dy
\end{equation*}
by weak convergence of $u_k$ to $u$, we have $\tilde u_j=\rho_j \star u$. Since the limit is unique, the entire sequence $\{\rho_j\star u_k\}_k$
converges to $\rho_j \star u$ in $L^p(A,\R^n)$.
Hence, for all $j$, there exists $k_j$ such that for all $k\geq k_j$,
\begin{equation}\label{eq:triang_2}
\|\rho_j \star u_k-\rho_j \star u\|_{L^p(A,\R^n)}\,\leq \, f(j^{-1}).
\end{equation}

By the triangle inequality, for $j$ large enough and for all $k$,
\begin{equation*}
\|u_k-u\|_{L^p(A,\R^n)}\,\leq\,\|u_k-\rho_j\star u_k\|_{L^p(A,\R^n)}+\|\rho_j\star u_k - \rho_j \star u\|_{L^p(A,\R^n)}+\|\rho_j\star u-u\|_{L^p(A,\R^n)}.
\end{equation*}
Hence, by \eqref{eq:triang_1}, \eqref{eq:triang_2}, and \eqref{eq:triang_3}, for all $j$ large enough and all $k\geq k_j$,
\begin{equation*}
\|u_k-u\|_{L^p(A,\R^n)}\,\leq\,3f(j^{-1})+\e_k.
\end{equation*}
This implies that $u_k\to u$ in $L^p(A,\R^n)$.
\end{proof}

In the proof of Theorem~\ref{th:approx1} we shall make use of the following
$\Gamma$-compactness and convergence of infima results, for
which we use terminology from Definition \ref{quasidef}.
\begin{lemma}\label{lem:compact}
%Let $0 < \rho_1 < \rho_2 < \infty$.
Let $\{\zeta_k\}_{k\in \N}$ be a sequence of point sets
in $\calA_{\rho_1,\rho_2}$.
% $(\rho_1,\rho_2)$-admissible point sets in $\R^d$.
% and let $\calT(\zeta_k)$ be associated Delaunay
%tessellations of $\R^d$.
%Let $p>1$ and let $f_{nn}$  and $W_\vol$ be of 
%class $\mathcal{U}_p$ and $\mathcal{V}_p$, respectively.
Let $D\in \calO(\R^d)$ be an open set.
% and for all $\e>0$ and all $A\in \calO(D)$, let
%$F_\e^A(\zeta_k)$ be the energy functional given by
%(\ref{eq:def-ener-vol2}) with $\zeta_k$ in place of $\calL$. 
Then for any sequence $\{\e_k\}_{k\in \N}$ of positive
numbers converging to zero, 
there exist a subsequence (not relabeled) and a Carath\'eodory function $W:D\times \M^{n\times d}\to \R^+$
such that for all open sets $A\in \calO(D)$ the functionals 
$F_{\e_k}^A(\zeta_k)$ 
(given by (\ref{eq:def-ener-vol2})) satisfy
\begin{equation}
\Gamma-\lim_{k\to \infty} F_{\e_k}^A(\zeta_k) = F^A
\label{111121}
\end{equation}
where the integral
functional $F^A:L^{p}(D,\R^n)\to[0,+\infty]$ is defined by
\begin{equation}
\label{FAdef}
F^A(u)=\begin{cases}\int_A W(x,\nabla
u(x))\ud x & \text{if } u\in W^{1,p}(A,\R^n), \cr +\infty
&\text{otherwise}.
\end{cases}
\end{equation}
In addition, $W$ is quasiconvex in its second variable and
satisfies a standard growth condition \eqref{eq:def-sgc} of order $p$.
\end{lemma}
In this subsection we will only make use of this result for $D=Q$.
\begin{proof}[Proof]
Lemma~\ref{lem:compact} is a corollary of \cite[Theorem~3]{Alicandro-Cicalese-Gloria-07b}, which is the corresponding compactness result
when the sequence $\zeta_k$ is a constant sequence which corresponds to the realization of a stochastic lattice, and for $W_\vol\equiv 0$.
Yet, there, the energy functionals are defined on piecewise constant functions (on the Voronoi tessellation associated with the point set), rather than 
on continuous and piecewise affine functions on a Delaunay tessellation.
As pointed out in \cite[Remark~4]{Alicandro-Cicalese-Gloria-07b}, the compactness result of \cite[Theorem~3]{Alicandro-Cicalese-Gloria-07b} holds as well 
for energy functionals defined on continuous and piecewise affine functions on an associated Delaunay tessellation.
Hence, since the proof of the individual compactness result \cite[Theorem~3]{Alicandro-Cicalese-Gloria-07b}
only makes use of deterministic arguments, it carries over to the case dealt with here provided $W_\vol\equiv 0$.
It remains to argue that one can add the volumetric term $W_\vol$ in \cite[Theorem~3]{Alicandro-Cicalese-Gloria-07b}.
Since this term yields a continuous contribution (it already has the form of an integral) and since the proof of \cite[Theorem~3]{Alicandro-Cicalese-Gloria-07b}
has the same structure as for the homogenization of multiple integrals (application of the De Giorgi-Letta criteria, and
of the integral representation results by Buttazzo and Dal Maso), this additional contribution can be treated in a standard way (see for instance \cite[Chapter~12]{Braides-98}).
\end{proof}
Recall notation $D_{R,r}$ from Section \ref{sec:result}. 
For $\e >0$, $r  > 0$,  $A \in \calO(\R^d)$,
and $\zeta \in \calA_{\rho_1,\rho_2}$, and $\Lambda \in \M^{n\times d}$,
define
%\begin{multline}\label{eq:def-space-BC}
\begin{equation}\label{eq:def-space-BC}
\mathfrak{S}_{\e}^{A,\Lambda,r}(\zeta)
:=
\Big\{ u\in  
{\mathfrak S}_{\e}^{A} (\zeta)
\big|
u \equiv \varphi_\Lambda \mbox{ on }
(\e \zeta \cap A \setminus A_{1,\e r}) \cup \partial A \Big\}.
\end{equation}
For all tessellations $\calT$ of $\R^d$, and all $ A\in \calO(\R^d)$,
we define $A_{\calT}$ as the interior of the closed set
$\cup_{T\in \calT, T\subset A}\overline{T}$.

\begin{lemma}
\label{lem1211}
Suppose $A,A',A_+ \in \calO(\R^d)$ are open sets with
the closure of $A$ contained in $A'$ and
the closure of $A'$ contained in $A_+$.
Let $\Lambda \in \M^{n\times d}$ and let $r \geq 2 \rho_2$.
Then there exists  a constant $C >0$ such that for
all $\zeta \in \calA_{\rho_1,\rho_2}$, $\e \in (0,1)$, and
$u \in \mathfrak{S}_{\e}^{A,\Lambda,r}(\zeta)$,
with  $u$  extended by $\varphi_\Lambda$ on $A_+\setminus A$, and
all $h\in \R^d$ with $|h| < 1/C$, we have
\begin{equation*}
\|\tau_h u-u\|^p_{L^p(A',\R^n)} \,\leq\, C
[ |h|^p (1+ F^{A}_{\e}(\zeta,u) ) + \e ],
\end{equation*}
where $\tau_h$ is as defined in Lemma \ref{lem:Frechet-Kolmogorov}.
\end{lemma}
\begin{proof}[Proof.]
By \cite[formula (31)]{Alicandro-Cicalese-Gloria-07b} 
(the result can be extended to the case $W_\vol\geq 0$, since the associated
contribution to the energy
is non-negative), 
there is a constant $C$ such that
for all $h\in \R^d$ with $|h| < 1/C$,
\begin{equation}
\label{1209a}
\|\tau_h u-u\|^p_{L^p(A',\R^n)} \,\leq\, C[
|h|^p
(F_{\e}^{A_+}(\zeta,u)+ 1)
+\e ]
\end{equation}
for all $\zeta,\e,u$.
By definition of the functionals and nonnegativity of $f_{nn}$ and $W_\vol$,
\begin{eqnarray}
F_{\e}^{A_+}(\zeta,u)
&\leq & 
F_{\e}^{{A}_{\calT(\e \zeta)}}
(\zeta,u)+F_{\e}^{A_+\setminus \overline{{A}_{\calT(\e \zeta)}}}
(\zeta,u)
\nonumber
\\
\label{1209b}
& \leq &
F_{\e}^{A}
(\zeta,u)+F_{\e}^{A_+\setminus \overline{{A}_{\calT(\e \zeta)}}}
(\zeta,u).
\end{eqnarray}
Since $r \geq 2 \rho_2$,
for all $k$,  all cells of
$\calT(\e \zeta)$ which have a vertex in $A_{1, \e r}$
are contained in $A$ so that
for $u \in \mathfrak{S}_{\e}^{A,\Lambda,r}(\zeta)$ we have
\begin{eqnarray*}
%\label{1209d}
u \equiv
\, {\varphi_\Lambda }
~~~ \mbox{ on }
~~~
A_+\setminus A_{\calT(\e \zeta)}.
\end{eqnarray*}
Hence by non-negativity of $f_{nn}$ and $W_\vol$, 
\begin{equation*}
F_{\e}^{A_+\setminus \overline{A_{\calT(\e \zeta)}}}
(\zeta,u) \,=\,F_{\e}^{A_+\setminus \overline{A_{\calT(\e \zeta)}}}
(\zeta,\varphi_\Lambda)\,\leq\,F_{\e}^{A_+}(\zeta,\varphi_\Lambda).
\end{equation*}
By Lemma~\ref{lem:unif-bd}, the last expression is bounded
uniformly in $\zeta,\e,u$.
Combined with (\ref{1209a}) and (\ref{1209b}), this
gives us  result.
\end{proof}

\begin{lemma}\label{lem:BC}
In addition to the notation and assumptions of Lemma~\ref{lem:compact}, 
let $\Lambda \in \M^{n\times d}$ and
$A\in\calO(D)$ be an open set.
Let $W$ be as in Lemma \ref{lem:compact}.
Suppose $\e_k$ is a sequence such that
$\e_k >0$, $\lim_{k \to \infty} \e_k = 0$ and
(\ref{111121}) holds.
Then for $r \geq 2 \rho_2$,
\begin{equation*}
\lim_{k\to \infty}\inf\left\{  F^A_{\e_k}(\zeta_k,u)\,|\, u\in {\mathfrak{S}}_{\e_k}^{A,\Lambda,r}(\zeta_k)\right\} \,
=\,\inf_{u\in W^{1,p}_0(A)}
\left\{ \int_A W(x,\Lambda+\nabla u(x))dx
%, \,u\in W^{1,p}_0(A)
\right\}.
\end{equation*}
%
%where 
%
%
\end{lemma}
\begin{proof}[Proof.]
Let $r \geq 2 \rho_2$. For $\e >0$ and
$\zeta \in \calA_{\rho_1,\rho_2}$,
define the functionals
$F^{A,\Lambda,r}_{\e}(\zeta,\cdot):L^p(A,\R^n)\to [0,+\infty]$
and $F^{A,\Lambda}:L^p(A,\R^n)\to [0,+\infty]$ by
\begin{equation*}
F^{A,\Lambda,r}_{\e}(\zeta,u)\,=\,\left\{
\begin{array}{rl}
F^{A}_{\e}(\zeta,u) &\mbox{ if }u\in \mathfrak{S}_{\e}^{A,\Lambda,r}(\zeta),\\
+\infty&\mbox{ else};
\end{array}
\right.
\end{equation*}
\begin{equation*}
F^{A,\Lambda}(u)\,=\,\left\{
\begin{array}{rl}
F^{A}(u) &\mbox{ if }u-\varphi_\Lambda\in W^{1,p}_0(A,\R^n),\\
+\infty&\mbox{ else}.
\end{array}
\right.
\end{equation*}

We split the proof into two steps.
First we prove the $\Gamma$-convergence result
\begin{equation}
\Gamma-\lim_{k\to \infty} F_{\e_k}^{A,\Lambda,r}(\zeta_k)
= F^{A,\Lambda}.
\label{1208a}
\end{equation}
Then we appeal to Lemma~\ref{lem:conv-min} to prove the desired result.

\step{1}{Proof of \eqref{1208a}}
Let $A'\in \calO(\R^d)$ be an open set which contains
the closure of $A$, and let
$A_+ \in \calO(\R^d)$ be an open set which contains the closure of
$A'$.

To prove (\ref{1208a}), we  first address the $\Gamma$-liminf inequality.
Let $u_k\in \mathfrak{S}_{\e_k}^{A,\Lambda,r}(\zeta_k)$ be a 
sequence converging to some $u $ in $ L^p(A,\R^n)$
such that 
\begin{equation}\label{eq:assump-bded}
\sup_k F^{A,\Lambda,r}_{\e_k}(\zeta_k,u_k)\,<\,+\infty.
\end{equation}
Since $F^{A}_{\e_k}(\zeta_k,\cdot)\leq F^{A,\Lambda,r}_{\e_k}(\zeta_k,\cdot)$, the 
assumption (\ref{111121}) and estimate \eqref{eq:assump-bded}
%$\Gamma$-convergence of $F^{A}_{\e_k}(\zeta_k)$ to $F^A$
imply that 
\begin{equation*}
F^A(u) \,\leq\,\liminf_{k\to \infty}F^{A}_{\e_k}(\zeta_k,u_k) \,<\,+\infty,
\end{equation*}
so that $u\in W^{1,p}(A,\R^n)$. For this to imply the desired $\Gamma$-liminf inequality, we need to prove that $u-\varphi_\Lambda\in W^{1,p}_0(A,\R^n)$ so that $F^{A,\Lambda}(u)=F^A(u)$.

%By \cite[formula (31)]{Alicandro-Cicalese-Gloria-07b} 
%(the result can be extended to the case $W_\vol\geq 0$, since the associated
% contribution to the energy is non-negative), 
%for all $h\in \R^d$ with $|h|$ small enough,
%\begin{equation*}
%\|\tau_h u_k-u_k\|^p_{L^p(A')} \,\leq\, C[
% (F_{\e_k}^{A_+}(\zeta_k,u_k)+|A_+|)|h|^p+\e_k ]
%\end{equation*}
%for all $k$.
%By definition of the functionals and nonnegativity of $f_{nn}$ and $W_\vol$,
%\begin{equation*}
%F_{\e_k}^{A_+}(\zeta_k,u_k)\,\leq 
%\,F_{\e_k}^{{A}_{\calT(\e_k \zeta_k)}}
%(\zeta_k,u_k)+F_{\e_k}^{{A_+\setminus {A}_{\calT(\e_k \zeta_k)}}}
%(\zeta_k,u_k).
%\end{equation*}
%We shall prove that both contributions are uniformly bounded
% with respect to $k$.
%The first term is bounded by assumption \eqref{eq:assump-bded}, noting that 
%\begin{equation*}
%F_{\e_k}^{{A}_{\calT(\e_k \zeta_k)}}(\zeta_k,u_k)\,=\,
%F_{\e_k}^{A}(\zeta_k,u_k)\,=\,F^{A,\Lambda,r}_{\e_k}(\zeta_k,u_k).
%\end{equation*}
%Since $r >2 \rho_2$,
% for all $k$,  all Delaunay cells intersecting $A_{1,r}$
%are contained in $A$ so that
%\begin{equation}\label{eq:bc-uk}
%{u_k}
%\equiv
%\, {\varphi_\Lambda }
%~~~ \mbox{ on } ~~~
%A_+\setminus A_{\calT(\e_k \zeta_k)}.
%\end{equation}
%%
%so that
%%
%\begin{equation*}
%F_{\e_k}^{A_+\setminus {A}_{\calT(\e_k \zeta_k)}}
%(\zeta_k,u_k) \,=\,F_{\e_k}^{A_+\setminus {A}_{\calT(\e_k \zeta_k)}}
%(\zeta_k,\varphi_\Lambda)\,\leq\,F_{\e_k}^{A_+}(\zeta_k,\varphi_\Lambda)
%\end{equation*}
%%
%by non-negativity of $f_{nn}$ and $W_\vol$. 
%By Lemma~\ref{lem:unif-bd}, this term is bounded uniformly in $k$.
%Thus

Extend $u_k$ and $u$ by $\varphi_\Lambda$ on $A_+\setminus A$ for all $k$, so that $u_k\to u$ in $L^p(A_+)$.
By Lemma \ref{lem1211}, and (\ref{eq:assump-bded}),
there exists some $C>0$ such that for all $k$ and 
for all $h\in \R^d$ with $|h| < 1/C$,
\begin{equation}\label{eq:verif-translation}
\|\tau_h u_k-u_k\|^p_{L^p(A',\R^n)} \,\leq\, C( |h|^p+ \e_k). 
\end{equation}

% By Lemma \ref{lem:Frechet-Kolmogorov} there exists a subsequence such that
Since $\tau_hu_k-u_k\to \tau_hu-u$ in $L^p(A',\R^n)$,
% along the subsequence, and hence
%
\begin{equation*}
\|\tau_h u-u\|^p_{L^p(A',\R^n)} \,\leq\, C|h|^p
\end{equation*}
for all $h\in \R^d$ with $|h| <1/C$.
From the characterization of $W^{1,p}(A',\R^n)$ by difference quotients (see for instance
\cite[Theorem~3 Section 5.8]{Evans-98}) we thus deduce that $u\in W^{1,p}(A',\R^n)$. Combined with the fact that 
%$u|_{A'\setminus A}\equiv {\varphi_{\Lambda}}|_{A'\setminus A}$
$u \equiv \varphi_{\Lambda}$ on $A'\setminus A$, and the continuity of the 
trace operator from $W^{1,p}(A',\R^n)$
to $W^{1-1/p,p}(\partial A,\R^n)$
(see for instance \cite[Theorem~1 Section 5.5]{Evans-98}),
%for any $A'$ which strictly contains $A$,
this shows that $u=\varphi_{\Lambda}$ in $W^{1-1/p,p}(\partial A,\R^n)$, and therefore $u-\varphi_{\Lambda}\in W^{1,p}_0(A,\R^n)$. The $\Gamma$-liminf inequality is proved.

To prove the existence of recovery sequences, for every
$u$ with
$u - \varphi_\Lambda\in W^{1,p}_0(A,\R^n)$ we have to construct a sequence $u_k\in \mathfrak{S}_{\e_k}^{A,\Lambda,r}(\zeta_k)$
which converges to $u$ in $L^p(A,\R^n)$ and satisfies
\begin{equation*}
F^{A,\Lambda}(u) \,=\,\lim_{k\to \infty}F^{A,\Lambda,r}_{\e_k}(\zeta_k,u_k),
\end{equation*}
which we may also rewrite as $F^{A}(u) \,=\,\lim_{k\to \infty}F^{A}_{\e_k}(\zeta_k,u_k)$.
Using the $\Gamma$-convergence of $F^{A}_{\e_k}(\zeta_k)$, it
is enough to modify a recovery sequence for $u$ in
$\mathfrak{S}_{\e_k}^{A}(\zeta_k)$ so that it belongs to 
$\mathfrak{S}_{\e_k}^{A,\Lambda,r}(\zeta_k)$
and still satisfies the identity above.  This modification can be
achieved using De Giorgi's averaging method, which is a very technical
argument in this discrete case.  We refer the reader to 
\cite[Proof of Proposition~3]{Alicandro-Cicalese-Gloria-07b},
where this issue is treated in detail.

\medskip

\step{2}{Proof of the convergence of infima}
With (\ref{1208a}) established, we shall  complete the proof
by an application of  
Lemma~\ref{lem:conv-min}.
Let $u_k$ be a sequence of minimizers of $F^{A,\Lambda,r}_{\e_k}(\zeta_k)$ on $\mathfrak{S}_{\e_k}^{A,\Lambda,r}(\zeta_k)$.
We extend $u_k$ by $\varphi_\Lambda$ on $A_+\setminus A$ for all $k$.
Consider an arbitrary subsequence of the original
sequence $u_k$, with this 
subsequence also denoted $u_k$. 
Since the $u_k$ are minimizers,
by Lemma \ref{lem:unif-bd} we have that
\begin{equation}
\sup_k F_{\e_k}^{A,\Lambda,r} (\zeta_k,u_k) \leq
\sup_k F_{\e_k}^{A,\Lambda,r} (\zeta_k,\varphi_\Lambda) < \infty.
\label{1212a}
\end{equation}
To check the conditions of
Lemma~\ref{lem:conv-min}, 
we shall apply Lemma \ref{lem:Frechet-Kolmogorov}.
To check that we may do so, we need to
prove that $u_k$ is bounded in $L^p(A_+,\R^n)$.
Since $\zeta_k \in \calA_{\rho_1,\rho_2}$,
%
%\begin{equation*}
%\|u_k\|_{L^p(A_{\calT(\e_k \zeta_k)},\R^n)}^p
% \,\leq \,C  \e_k^d  \sum_{x\in {A}\cap \e_k\zeta_k}|u_k(x)|^p.
%\end{equation*}
%
%Indeed, since
and $u_k$ is affine on each simplex of $\calT(\e_k \zeta_k)$
contained in ${A}$, we claim that
there exists $C>0$ such that for all $k$
\begin{eqnarray*}
\|u_k\|_{L^p(A_{\calT(\e_k \zeta_k)},\R^n)}^p
&=&\sum_{T\in \calT(\e_k \zeta_k),T\subset A} 
\int_{T} |u_k(x)|^pdx \\
&\leq & \sum_{T\in \calT(\e_k \zeta_k),
T\subset {A}}  |T| \sup_{x \mbox{ vertex of }T} |u_k(x)|^p \\
&\leq & C \e_k^d \sum_{x \in \overline A \cap \e_k \zeta_k }|u_k(\e_k x)|^p. 
\end{eqnarray*}
%
%for some $C >0$ independent of $k$.
This claim follows from the following two facts:
\begin{itemize}
% \item simplices are closed sets so that if a simplex 
% $T \in \calT(\e_k \zeta_k)$
% is such that
% $ T\subset A$, then its associated vertices are also in $A$,
\item the volume of any simplex in $\calT(\e_k \zeta_k)$
is bounded by $C_d(2\rho_2\e_k)^d $
(where $C_d$ denotes the volume of the unit ball in dimension $d$),
\item the number of neighbours of any point in the tessellation is uniformly bounded (as proved already) so that each point $x \in \e_k \zeta_k \cap \overline{A}$ 
lies in the closure of at most a fixed number of Delaunay cells $T \in \calT(\e_k \zeta_k)$
(independent of $k$).
\end{itemize}

On the other hand, since $r \geq 2 \rho_2$,
for all $v\in \mathfrak{S}_{\e_k}^{A,\Lambda,r}(\zeta_k)$,
$(v -\varphi_\Lambda)|_{A_{\calT(\e_k \zeta_k)}}  
\in W^{1,p}_0(A_{\calT(\e_k \zeta_k)})$, and we claim that we have a discrete Poincar\'e's inequality in 
$\mathfrak{S}_{\e_k}^{A}(\zeta_k) \cap W^{1,p}_0(A_{\calT(\e_k \zeta_k)})$,
namely 
\begin{multline}\label{eq:disc-Poinc}
\sum_{x\in \overline A\cap \e_k\zeta_k}\e_k^d |v(x)-\varphi_\Lambda(x)|^p
\\\,\leq\, C \sum\limits_{\begin{array}{cc}
\{x,y\}\in \mathcal{N}(\e_k \zeta_k)\cr
(x,y)\subset \overline{A} \end{array}} \e_k^d\left(\frac{|v(y)-\varphi_\Lambda(y)-v(x)+\varphi_\Lambda(x)|}{ |y-x|}\right)^p.
%|A||\Lambda|^p,
\end{multline}
where $C$ depends on  $A$ but not on $k$.
The proof of \eqref{eq:disc-Poinc}  uses the fact,
%begins as for the standard proof of Poincar\'e's inequality using the
% fundamental theorem of calculus, up to
%replacing the derivative by finite differences. In particular, as
proved in \cite[Lemma~3]{Alicandro-Cicalese-Gloria-07b}, that for any
$x \in \overline A\cap \e_k\zeta_k$, one can find a path $\gamma(x)$ on
the Delaunay graph from $x$ to some point $x_0 \in \partial 
A_{\calT(\e_k \zeta_k)}\cap \e_k \zeta_k$
% near the boundary of $A$,
with $O(\e_k^{-1})$ steps, and moreover arrange that each Delaunay edge
appears in at most $O(\e_k^{-1})$ of the paths 
$\gamma(x)$ for $x \in \overline A\cap \e_k\zeta_k$.
Let $\gamma(x)=\{x_0,x_1,\dots,x_\ell,x\}$ for some $\ell \in \N$.
Then, for any 
$w\in \mathfrak{S}_{\e_k}^{A}(\zeta_k) 
\cap W^{1,p}_0(A_{\calT(\e_k \zeta_k)})$, 
\begin{equation*}
|w(x)|\,\leq\, |w(x_0)|+\sum_{j=1}^\ell |w(x_{j})-w(x_{j-1})| \,\leq 
\,C\sum_{j=1}^\ell \e_k \frac{|w(x_{j})-w(x_{j-1})|}{|x_j-x_{j-1}|}
\end{equation*}
where $x_{\ell +1}=x$, using in addition that the edge lengths are
of order $\e_k$.
Since the path has length of order $O(\e_k^{-1})$, Jensen's inequality yields
%
%\begin{equation}\label{eq:pr-Poinca2}
$$
|w(x)|^p\,\leq \,C\sum_{j=1}^\ell \e_k \left(\frac{|w(x_{j})-w(x_{j-1})|}{|x_j-x_{j-1}|}\right)^p.
$$
%\end{equation}
%
Putting  $w=v -\varphi_\Lambda$ and
summing over $x \in \overline A\cap \e_k\zeta_k$,
using that each Delaunay edge appears in at most
$O(\e_k^{-1})$ of the paths $\gamma(x)$, 
%\eqref{eq:pr-Poinca2}
yields
the desired Poincar\'e inequality \eqref{eq:disc-Poinc}.

By the triangle inequality and the uniform bound 
$|x-y|\leq 2\rho_2$ for any edge $(x,y)$, the inequality
(\ref{eq:disc-Poinc}) for $v=u_k$ yields
\begin{equation*}
\sum_{x\in \overline A\cap \e_k\zeta_k}\e_k^d |u_k(x)|^p
\,\leq\, C \sum\limits_{\begin{array}{cc}
\{x,y\}\in \mathcal{N}(\e_k \zeta_k)\cr
(x,y)\subset  \overline{A} \end{array}} \e_k^d\left(\frac{|u_k(y)-u_k(x)|}{|y-x|}\right)^p +C|\Lambda|^p.
%|A||\Lambda|^p,
\end{equation*}
Using the property \eqref{eq:gcSRL} of $f_{nn}$ and the definition of $F^A_{\e_k}(\zeta_k)$, we deduce that
\begin{eqnarray*}
\|u_k\|_{L^p(A_{\calT(\e_k \zeta_k)},\R^n)}^p &\leq
&  C\sum\limits_{\begin{array}{cc}
\{x,y\}\in \mathcal{N}(\e_k \zeta_k)\cr
( x, y)\subset \overline{A} \end{array}} 
\e_k^d f_{nn}\left(\e^{-1}(y-x),\frac{u_k( y)-u_k( x)}{ |y-x|} 
\right)+C
%|A||\Lambda|^p
\\
&\leq &C F^{A,\Lambda,r}_{\e_k}(\zeta_k,u_k)+C
%(1+|\Lambda|^p)
%\\
%&\leq &C F^{A,\Lambda,r}_{\e_k}(\zeta_k,\varphi_\Lambda)+C(1+|\Lambda|^p)
\end{eqnarray*}
for some $C>0$ independent of $k$.
%In the last line, we have used that $\varphi_\Lambda$ has a larger energy
% than $u_k$ since $u_k$ is a minimizer.
By \eq{1212a} 
% Lemma \ref{lem:unif-bd},
the last expression is bounded uniformly in $k$.
%Hence $u_k$ is a bounded sequence in $L^p(A_{\e_k \calT(\zeta_k)},\R^n)$.

Using Lemma \ref{lem1211},  followed 
by (\ref{1212a}),
we see that there are constants $C,C'$ such
that for $|h|<1/C$ and for all $k$ we have 
$$
\|\tau_h u_k-u_k\|_{L^p(A',\R^n)}^p
\leq C[ |h|^p (1+ F^A_{\e_k}(\zeta_k,u_k)) + \e_k ]
\leq C' [ |h|^p + \e_k].
$$
%\eqref{eq:translation} holds for
%$u_k$ on $A'$.
% \in \calO(\R^d)$ whose closure is
% contained in the interior of $A_+$,
Hence Lemma~\ref{lem:Frechet-Kolmogorov} is applicable, 
showing  that the sequence $u_k$ converges 
along some subsequence to a limit
in $L^p(A,\R^n)$. Thus the closure of 
the original sequence $u_k$ is sequentially compact
in $L^p(A,\R^n)$, and by the equivalence
of compactness and sequential compactness in any metric space,  
this sequence is precompact.
We are thus in position to apply Lemma~\ref{lem:conv-min}, which concludes the proof.
\end{proof}
%
%We now prove Theorem~\ref{th:approx1}.
%
\begin{proof}[Proof of Theorem~\ref{th:approx1}]
In order to apply Theorem ~\ref{th:sub}
to the problem under consideration, we need to first define the
function  $S$;  we also define a family of auxiliary functions
$\tilde{S}^r, r >0$. Fix $\Lambda \in \M^{n\times d}$.
For  $\zeta \in \calA_{\rho_1,\rho_2}$,
and for $\e,r >0$,
recall the definition (\ref{eq:def-space-BC}) 
of ${\mathfrak S} _\e^{A,\Lambda,r}(\zeta)$, and
define
\begin{eqnarray*}
S(Q_R,\zeta) & = &
\inf_u \{F_1^{Q_{R- 2 \rho_2}}(\zeta,u)
\big| 
u\in {\mathfrak S}^{Q_{R-2 \rho_2},\Lambda,2 \rho_2}_1
%u\in {\mathfrak S}^{Q_R}_1
%(\zeta),u \equiv \varphi_{\Lambda}  \mbox{ on }
%\zeta \cap   Q_{R-2 \rho_2} \setminus Q_{R-4 \rho_2}
\};
%\label{Sdef}
\\
\tilde S^r(Q_R,\zeta) & = &
\inf_u \{F_1^{Q_{R}}(\zeta,u)
\big| u\in {\mathfrak S}^{Q_R,\Lambda,r}_1
%(\zeta),u \equiv \varphi_{\Lambda}  \mbox{ on }
%\zeta \cap  Q_R \setminus Q_{R-r} \cup \partial Q_R
\}.
%%\label{tSdef}
\end{eqnarray*}
%and set $\tilde S(Q_R,\zeta) :=  \tilde S^{\gamma + 2 \rho_2}(Q_R,\zeta)$.
%
%

Let $\zeta, \zeta' \in \calA_{\rho_1,\rho_2}$
with $\zeta \cap Q_{R} = \zeta' \cap Q_{R} $. 
Then for any $T \in \calT(\zeta)$ with 
$T \cap Q_{R-2 \rho_2} \neq \emptyset$,
the  circumsphere of $T$ contains no
point of $\zeta$ in its interior by definition, so it has radius less than
$2 \rho_2$ and therefore is contained in $Q_R$;
hence it contains no point of $\zeta'$ and hence
$T \in \calT(\zeta')$ as well.
In short the Delaunay tessellations of $\zeta$
and $\zeta'$ coincide on $Q_{R - 2 \rho_2}$.

Consequently,
our $ S$ is local on $\calD(Q)$ because
changes to $\zeta$ outside $Q_R$ do not affect
$S(Q_R,\zeta)$.
Also it can readily be deduced from
Theorem  \ref{th:main-gen} that
$S$ satisfies the averaging property on $\calA_{\rho_1,\rho_2}$
with respect to $Q$.

It remains to prove that $S$ is insensitive to boundary effects.
Let $\alpha \in (0,1)$.
It is enough to prove that for any sequence $\{\zeta_k\}_{k\in \N}$ of
general $(\rho_1,\rho_2)$-admissible point sets and any sequence of positive
numbers $\{R_k\}_{k\in \N}$ tending to infinity,
one has
\begin{equation}
\limsup_{k\to \infty}\frac{|S(Q_{R_k},\zeta_k)- 
S(Q_{R_k,R_k^\alpha},\zeta_k)|}{|Q_{R_k}|}\,=\,0
\label{1124a}
\end{equation}
where we are using notation $D_{R,r}$ from Section
\ref{sec:result}.
In fact, we first prove this for $\tilde S^r$, for  $r= 4 \rho_2$.

Set $\e_k=1/R_k$ for all $k\in \N$.
By Lemma~\ref{lem:compact},
after extraction (we do not relabel $k$) there exists a
Carath\'eodory function
$W:Q\times \M^{n\times d} \to \R^+$ such that for every open set $A\in {\calO}(Q)$, 
(\ref{111121}) holds, with $F^A$ given
by (\ref{FAdef}), and such that
$W$  is quasiconvex in its second variable and 
satisfies a standard growth condition \eqref{eq:def-sgc} of order $p>1$ (where $p$ is as in Theorem~\ref{th:main-gen}). 

Next we rescale.  For every open set  $A\in \calO(Q)$, and 
for all $\e >0, r >0 $, 
define ${\mathfrak S}_{\e}^{A,\Lambda,r}(\zeta_k)$  by
\eqref{eq:def-space-BC}
and define the rescaling $\tilde S^{\e,r}$ of $\tilde S^r$ by
\begin{eqnarray*}
\tilde S^{\e,r}(A,\zeta_k)
%&:=&\inf \{\tilde F^(\zeta)(u),u\in \tilde \mathfrak{S}^R(\zeta),
%u(x)=\Lambda\cdot x \mbox{ if }d_\infty(x,\partial Q_R)\leq r\} \\
&=& \inf_u \{F_\e^{A}( \zeta_k,u)
\big| u\in {\mathfrak S}_{\e}^{A,\Lambda,r}(\zeta_k) \}.
\end{eqnarray*}
If $R >0 $ and $u \in C(Q_R)$ and if $v \in C(Q)$
is defined by $v(x)= R^{-1} u(Rx)$, then 
for any $\zeta \in \calA_{\rho_1,\infty}$
we have
$$
F_{1/R}^{Q}(\zeta,v)= R^{-d} F_1^{Q_R}(\zeta,u),
$$
and moreover  $u \in \frakS_1^{Q_R,\Lambda, r}(\zeta) $
if and only if $v \in \frakS_{1/R}^{Q,\Lambda,r}(\zeta) $ for
any $r>0$.
Hence for any $r >0$,
\begin{eqnarray}
\tilde S^r(Q_{R},\zeta) & = & 
\inf \{ R^d F_{1/R}^{Q}(\zeta,v) \, \big| \,
v \in \frakS_{1/R}^{Q,\Lambda,r} ( \zeta)
%, v \equiv \varphi_{\Lambda}
%{\rm ~on ~}
%R^{-1} \zeta \cap Q \setminus Q_{1- r/R} 
\}
\nonumber
\\
%& = & R^d S^{1/R,(\gamma +2) \rho_2}(Q,\zeta_k).
& = & R^d \tilde S^{1/R,r}(Q,\zeta_k).
\label{eq:nontrivial}
\end{eqnarray}

Next we claim that if $(r_k)_{k \in \N}$ is any sequence
of numbers satisfying $r_k \geq 4 \rho_2$ and
$r_k = o(R_k)$ as $k \to \infty$,
then
\begin{equation}\label{eq:pr-approx-2.1}
\limsup_{k\to \infty}
\left| \tilde S^{\e_k,r_k}(Q,\zeta_k)
-
R_k^{-d} \tilde S^{4 \rho_2}(Q_{R_k},\zeta_k)
\right|
\,=\, 0.
\end{equation}
%
%(note that the second limit exists by \eq{0819b}). 
%
To prove this, first we deduce
from Lemma~\ref{lem:BC}, that for any $r\geq 2 \rho_2$
and any open set $A \in \calO(Q)$, we have
\begin{equation}\label{eq:pr-approx-2.1-0}
\lim_{k\to \infty} \tilde S^{\e_k,r}(A,\zeta_k)
\,=\,\inf_{u\in W^{1,p}_0(A)}
\left\{ \int_{A}W(x,\Lambda+\nabla u(x))dx
%u\in W^{1,p}_0(A)
\right\}.
\end{equation}

Let $\eta >0$.  For all $r\geq s \geq 4 \rho_2$ and all $\e>0$,
${\mathfrak S}_{\e}^{Q,\Lambda,r}(\zeta_k)
\subset {\mathfrak S}_{\e}^{Q,\Lambda,s}(\zeta_k)$,
so that by definition of $\tilde S^{\e,r}$ and
\eqref{eq:nontrivial}, 
for all $k\in \N$
we have that
\begin{equation}\label{eq:nontrivial1}
R_k^{-d} \tilde S^{4 \rho_2}(Q_{R_k},\zeta_k)
\,=\,
\tilde S^{\e_k,4 \rho_2}(Q,\zeta_k)
\,\leq \,
\tilde S^{\e_k,r_k}(Q,\zeta_k) .
\end{equation}
%On the other hand,
Also,
by the non-negativity of the energy functions $f_{nn}$ and $W_\vol$,
for all $k \in \N$ such that $r_k \leq \eta R_k$,
\begin{eqnarray}
\tilde S^{\e_k,r_k}(Q,\zeta_k)
& \,\leq \, & 
\tilde S^{\e_k,\eta R_k}(Q,\zeta_k) 
%\end{equation*}
%
%By definition of $\tilde S^{\e,r}$ and  $F^{\overline{A}}_\e$, and
% if we set $\varphi_\Lambda:x\mapsto \Lambda \cdot x$,
%this inequality turns into
%
%\begin{equation}
%\frac{
%\tilde S^{\e_k,r_k}(Q,\zeta_k)
\nonumber \\
& \,\leq\, &
F^{Q\setminus \overline{Q_{1-2\eta}}}_{\e_k}(\zeta_k,\varphi_\Lambda)
+
\tilde S^{\e_k,2 \rho_2}(Q_{1-\eta},\zeta_k)
\label{eq:nontrivial2}
\end{eqnarray}
provided all the simplices of
$ \calT(\e_k \zeta_k)$ which
intersect $Q_{1-\eta}$ without being contained in $Q_{1-\eta}$ 
are contained in $Q\setminus Q_{1-2\eta}$ --- which holds provided
$2 \rho_2 \e_k < \eta$ since the
edge lengths of the simplices of 
$ \calT(\zeta_k)$ are bounded by $2 \rho_2$.

Lemma~\ref{lem:unif-bd} yields the bound 
$F^{Q\setminus  Q_{1-2\eta}}_{\e_k}(\zeta_k,
\varphi_\Lambda)\leq 
C|Q\setminus  Q_{1-2\eta}|(1+|\Lambda|^p)\leq C' \eta $.
Hence  by \eqref{eq:nontrivial1} and \eqref{eq:nontrivial2},
and (\ref{eq:nontrivial}) 
and property \eqref{eq:pr-approx-2.1-0} applied to $Q$ and to 
$Q_{1-\eta}$, 
% \Comment{MP} [replaced
% $Q\setminus Q_{1-\eta}$ here and
% repeatedly in next few lines. Also, taking the limit of the
% second term in the next display we seem to use the
% case  $r=0$ of \eq{eq:pr-approx-2.1-0}
% but it says there that it only applies for
% $r \geq 2\gamma \rho_2$. Please can you clarify]
%$Q\setminus Q_{1-\eta}$,
this implies that
\begin{multline*}
\limsup_{k\to \infty}\left|
\tilde S^{\e_k,r_k}(Q,\zeta_k)
- R_k^{-d} \tilde S^{4 \rho_2}(Q_{R_k},\zeta_k)  
\right|
%\\ 
\,\leq \, 
C\eta 
\\
+\inf_{u\in W^{1,p}_0 (Q_{1-\eta}) }
\left\{ 
\int_{Q_{1-\eta}}W(x,\Lambda+\nabla u(x))dx
%u\in W^{1,p}_0
%\big(
%Q_{1-\eta}
%\big)
\right\}
%\\
-\inf_{u\in W^{1,p}_0(Q)}
\left\{ \int_{Q}W(x,\Lambda+\nabla u(x))dx
%, u\in W^{1,p}_0(Q)  
\right\}.
\end{multline*}
We claim that the right hand side  of this
inequality tends to zero
as $\eta$ vanishes. It is enough to prove that the limsup of the difference of the infimum problems is non-positive.
By the properties of $W$, $u \mapsto \int_{Q}W(x,\Lambda+\nabla u(x))dx$ is continuous on $W^{1,p}_0(Q)$.
Let $\eta_k$ be a sequence of positive numbers which tends to zero.
Since $C^\infty_0(Q)$
is dense in $W^{1,p}_0(Q)$ by definition, there exists
a sequence $u_k \in C^\infty_0(Q)$ such that for all $k\in \N$,
$\supp(u_k)\subset  Q_{1-2 \eta_k}$
and which satisfies
\begin{equation*}
\lim_{k\to\infty} \int_{Q}W(x,\Lambda+\nabla u_k(x))dx\,=\,
\inf_{u\in W^{1,p}_0(Q)}
\left\{ \int_{Q}W(x,\Lambda+\nabla u(x))dx
%,u\in W^{1,p}_0(Q)
\right\}.
\end{equation*}
Since $W$ is non-negative and $u_k \in W^{1,p}_0
\big( Q_{1-\eta_k}\big)$ for all $k\in \N$, we have
\begin{equation*}
\inf_{u\in W^{1,p}_0 ( Q_{1-\eta_k} )}
\left\{ \int_{Q_{1-\eta_k}}W(x,\Lambda+\nabla u(x))dx,
%u\in W^{1,p}_0 \big( Q \big)
\right\}
\,\leq \,  \int_{Q}W(x,\Lambda+\nabla u_k(x))dx,
\end{equation*}
so that 
\begin{multline*}
\limsup_{k\to \infty} 
\inf_{u\in W^{1,p}_0 (Q_{1-\eta_k})}
\bigg\{ 
\int_{ Q_{1-\eta_k}}W(x,\Lambda+\nabla u(x))dx
%,u\in W^{1,p}_0
%\big(Q_{1-\eta_k}\big)
\bigg\} \\
\,\leq\,\inf_{u\in W^{1,p}_0(Q)}  
\left\{ \int_{Q}W(x,\Lambda+\nabla u(x))dx
%,u\in W^{1,p}_0(Q)  
\right\},
\end{multline*}
which is the claim.  This proves \eqref{eq:pr-approx-2.1}.

\medskip

Now observe that $Q_{R,R^\alpha} = Q_{R-R^\alpha}$. 
%for some function $h(R)$ \Comment{MP} [Provided $Q^*$ defined open] 
%satisfying $R^\alpha \leq h(R) \leq R^\alpha +1$.
The test-functions in the infimum problems defining 
$\tilde S^{R^\alpha+4 \rho_2}(Q_R,\zeta_k)$ and 
$\tilde S^{4 \rho_2}(Q_{R-R^\alpha},\zeta_k)$ coincide on $Q_{R-R^\alpha}$,
since they do coincide on $\{T\in  \calT(\zeta_k)
\,\big|\,T \cap Q_{R-R^\alpha -4 \rho_2}\neq \emptyset\}$,
and take the value $\varphi_\Lambda$ 
elsewhere on $Q_{R-R^\alpha}$.
Hence, taking into account Definition~\ref{hypo:ener-SRL}, this yields for all $R>0$,
\begin{equation}\label{eq:pr-approx-2.3}
\tilde S^{R^\alpha +4 \rho_2}
(Q_R,\zeta_k)\,=\,
\tilde S^{4 \rho_2}(Q_{R-R^\alpha},\zeta_k)+R^{d-1}O(R^\alpha).
\end{equation}
Using 
\eqref{eq:pr-approx-2.1}, \eqref{eq:nontrivial} and \eqref{eq:pr-approx-2.3}
in succession then gives us
for $r = 4 \rho_2$
that
\begin{eqnarray*}
\tilde S^{r}(Q_{R_k}, \zeta_k) & = & 
R_k^d \tilde S^{\e_k, R_k^\alpha +r} (Q,\zeta_k) + o(R_k^d)
\\
& = & \tilde S^{R_k^\alpha+r} (Q_{R_k},\zeta_k) + o(R_k^d) 
\\
& = & \tilde S^r(Q_{R_k - R_k^\alpha},\zeta_k) + o(R_k^d) 
\end{eqnarray*}
and therefore (\ref{1124a}) holds  with $S$ replaced  by $\tilde S^{4 \rho_2}$
along the subsequence, and hence along the original sequence too.
To demonstrate \eq{1124a} for $S$ itself, note that
for any $R >4 \rho_2$ and $\zeta \in \calA_{\rho_1,\rho_2}$,
if
$u \in  \mathfrak{S}^{Q_{R-2 \rho_2},\Lambda,2\rho_2}_1(\zeta)$
then extending $u$ by $\varphi_\Lambda$ to $Q_R \setminus Q_{R -2 \rho_2}$
gives a function (also denoted $u$) in 
$ \mathfrak{S}^{Q_{R},\Lambda,4\rho_2}_1(\zeta)$, while conversely
%with $u \equiv \varphi_\Lambda$ on $\zeta \cap Q_{R-2 \rho_2}
%\setminus Q_{R- 4 \rho_2}$, and
if $u \in \mathfrak{S}^{Q_R,\Lambda, 4\rho_2}(\zeta)$ then
$u|_{Q_{R- 2 \rho_2}}  \in \mathfrak{S}^{Q_{R- 2 \rho_2},\Lambda, 2\rho_2}
(\zeta)$.  
%by a similar argument to the one leading up to
%\eq{eq:pr-approx-2.3}
Hence there is a constant $C>0$ such that
in both cases
we have 
\begin{eqnarray*}
0 \leq F_1^{Q_{R}}(\zeta,u) - F_1^{Q_{R- 2 \rho_2}}(\zeta,u)
\leq C R^{d-1},
%\\
%0 \leq F_1^{Q_{R}}(\zeta,v) - F_1^{Q_{R- 4 \rho_2}}(\zeta,v)
%\leq C R^{d-1};
\end{eqnarray*}
and hence 
$|S(Q_R,\zeta) - \tilde S^{4\rho_2}(Q_R,\zeta)|
\leq C R^{d-1}$. Then it is clear that  \eq{1124a} for $S$ follows
from \eq{1124a} for $\tilde S^{4\rho_2}$.

We are thus in position to apply
% Lemma~\ref{lem:suff-averaging}
Theorem \ref{th:sub},
which concludes the proof of Theorem~\ref{th:approx1}.
\end{proof}

\subsection{Rubber elasticity and random parking in a bounded set}\label{sec:bded-set}

This subsection is devoted
to a second version of the second question of the introduction,
whereby we consider a more general domain D
than the cubes considered in Section~\ref{sec:ass-SET}.
Let $D\in \calO(\R^d)$ be some fixed open domain.  
We do not focus here on the approximation of $W_\ho$, but rather on the $\Gamma$-convergence result itself.
We define the discrete model as follows.
For all $\e>0$, we let $\xi^{1/\e}: = \xi^{D_{1/\e}}$
be the random parking measure with
parameter $\rho_0>0$ on
$D_{1/\e}:=\{\e^{-1} x, x\in D\}$, let $V_\e$ be the associated
Voronoi tessellation of $\R^d$, and let $ \calT_\e$
be the associated Delaunay tessellation of the convex hull
of $\xi^{1/\e}$. Although $\xi^{1/\e}$ is $(\rho_0/2,2\rho_0)$-admissible,
some edges of $ \calT_\e$
may be arbitrarily large.
In order to avoid such difficulties (which are
physically irrelevant anyway), we proceed as in the previous section and focus 
on $D_{1,4\e \rho_0}=\{x\in D\,|\,d_2(x,\partial D)\geq 4\e \rho_0\}$.
(Note that in the previous section $D=Q$ so that $D_{1,4\e \rho_0}=Q_{1,4\e \rho_0}=Q_{1-4\e\rho_0}$.)
By the geometric argument of the proof of Theorem~\ref{th:approx1}, $D_{1,4\e \rho_0}$
has the property that every simplex $T$ of the Delaunay tessellation  $\calT_\e$ is such that
if $\e T \cap D_{1,4\e \rho_0}\neq \emptyset$ then $\e T \subset D$, and therefore 
all the simplices of $\calT_\e$ contained in $D_{\e^{-1},4\rho_0}$ have uniformly bounded edge lengths.

Put $\rho_1 = \rho_0/2, \rho_2 = 2 \rho_0$, and
fix $D$.
For each $\e >0$ and $\zeta \in \calA_{\rho_1,\rho_2}(D_{1/\e})$,
and open $A \in \calO(D)$,
we define an energy functional 
$\overline F^A_\e(\zeta) := \overline F^A_\e(\zeta,\cdot) $ on 
 $L^p(A)$
as:
\begin{equation}\label{eq:def-F-var2}
\overline F_\e^A(\zeta,u)\,:=\,
F_\e^{A\cap D_{1,2\e \rho_2}}(\zeta,u) ,
\end{equation}
where $F_\e^A(\zeta,\cdot)$ is defined as in (\ref{eq:def-ener-vol2}).
%
% consider a sub-tessellation $\calT_\e$ of
% $\calT(\xi^{\e^{-1}})$ that we define as follows:
% %
% \begin{multline}\label{eq:def_T_rp}
% \calT_\e \,=\,\{T\in \calT(\xi^{\e^{-1}}) 
% \,|\, \mbox{ no vertex of }T\mbox{ is the center of Voronoi cell }C\\\mbox{ such that }C\cap (\R^d\setminus D_{\e^{-1}})\neq \emptyset\}.
% \end{multline}
% %
% We denote by $\tilde \xi^{1/\e}$ the set of vertices of $\calT_\e $, and define the set of neighbours as
% %
% \begin{equation}\label{eq:def_N_rp}
% \calN_\e = \{(x,y)\,|\,\exists \,T\in \calT_\e, \, [x,y] 
%\mbox{ is an edge of }T\}.
% \end{equation}
% %
% Elementary geometric arguments show that:
% %
% \begin{itemize}
% \item $\{x\in D\,|\,d_\infty(x,\partial D)\geq 4\rho_0 \e\} \subset \cup_{T\in \calT_\e}\e \overline T$, so that almost surely
% %
% \begin{equation*}
% \lim_{\e \to 0}|D\setminus \cup_{T\in \calT_\e}\e T|\,=\,0.
% \end{equation*}
% %
% \item the distance between any two neighbours is at most $4\rho_0 \e$.
% \end{itemize}
% %
% Hence $\tilde \xi^{1/\e}$ is $(\rho_0/2,4\rho_0)$-admissible in $D_{\e^{-1}}$.
%
%\medskip
%
We then have the following result:
% %
% \begin{theorem}\label{th:rp-gen}
% Let $D\in \calO(\R^d)$ be an open set and let $\rho_0>0$.
% For all $\e>0$ we let $\xi^{\e^{-1}}$ be the random parking
% measure with parameter $\rho_0$ in $D_{\e^{-1}}$, $\calT_\e$ be the Delaunay tessellation defined in \eqref{eq:def_T_rp}, and 
% $\calN_\e$ be the associated set of neighbours \eqref{eq:def_N_rp}.
% Denote by $F_\e^D(\xi^{\e^{-1}})$ the energy functional given by
% (\ref{eq:def-ener-vol2}), where in \eqref{eq:affine}, \eqref{def:ener-nn-lr}, and \eqref{eq:def-volume}, $T(\mathcal{L})$ and 
% $\mathcal{N}(\mathcal{L})$ are replaced by $\calT_\e$ and $\calN_\e$, respectively. 
% Assume that $f_{nn}$  and $W_\vol$ are of class $\mathcal{U}_p$ and $\mathcal{V}_p$ for some $p>1$.
% Then  the functionals $F_\e^D(\xi^{\e^{-1}})$ 
% $\Gamma$-converge as $\e\to 0$ to the deterministic integral
% functional $F_{\ho}^D:L^{p}(D,\R^n)\to[0,+\infty]$ defined by
% %
% \begin{equation*}
% F_{\ho}^D(u)=\begin{cases}\int_D W_{\ho}(\nabla
% u(x))\ud x & \text{if } u\in W^{1,p}(D,\R^n), \cr +\infty
% &\text{otherwise},
% \end{cases}
% \end{equation*}
% %
% where $W_{\ho}$ coincides with the energy density defined by the asymptotic homogenization
% formula \eqref{eq:asymp-fo-de} of Theorem~\ref{th:main-gen} when the 
% random point process $\calL$ is the random parking measure $\xi$ 
% with parameter $\rho_0$ in $\R^d$.
% \end{theorem}
% %
%
\begin{theorem}\label{th:rp-gen}
%Let $D\in \calO(\R^d)$ be an open set and let $\rho_0>0$.
%For $\e>0$ let $\xi^{1/\e}$ be the random parking
%measure with parameter $\rho_0$ in $D_{1/\e}$.
%Denote by $\overline{F}_\e^D$ the energy functional given by \eqref{eq:def-F-var2}.
Assume that $f_{nn}$  and $W_\vol$ are of class $\mathcal{U}_p$ and $\mathcal{V}_p$ for some $p>1$.
Then  the functionals $\overline{F}_\e^D(\xi^{1/\e})$ 
$\Gamma$-converge as $\e\to 0$ to the deterministic integral
functional $F_{\ho}^D:L^{p}(D,\R^n)\to[0,+\infty]$ defined by
\begin{equation*}
F_{\ho}^D(u)=\begin{cases}\int_D W_{\ho}(\nabla
u(x))\ud x & \text{if } u\in W^{1,p}(D,\R^n), \cr +\infty
&\text{otherwise},
\end{cases}
\end{equation*}
where $W_{\ho}$ coincides with the energy density defined by the asymptotic homogenization
formula \eqref{eq:asymp-fo-de} of Theorem~\ref{th:main-gen} when the 
random point process $\calL$ is the random parking measure $\xi$ 
with parameter $\rho_0$ in $\R^d$.
\end{theorem}
In particular, combining Theorems \ref{th:rp-gen}  
and \ref{th:main-gen}
shows that considering the restriction
of the random parking measure $\xi$ to $D_{1/\e}$ 
or considering the random parking measure $\xi^{1/\e}$ on $D_{1/\e}$ yield the same thermodynamic limit
for the discrete model of rubber dealt with in this section.

To prove this result, we need 
the following localized version of
Lemma~\ref{lem:compact}
for point sets defined on bounded domains, which is
proved similarly to 
Lemma~\ref{lem:compact}; 
%\Comment{AG} {\bf [This remark
%about the proof is correct. Yet this next lemma cannot be
%deduced from Lemma \ref{lem:compact}. The proof is however the same, provided
%one notices that
taking $D_{1,4\e\rho_0}$ instead of $D$ does not change
the estimates in
\cite{Alicandro-Cicalese-Gloria-07b}
since they are all set on domains compactly included in $D$.
\begin{lemma}\label{lem:compact-local}
Let $D\in \calO(\R^d)$ be an open set, and $\{\e_k\}_{k\in \N}$ be a sequence of positive
numbers converging to zero.
Let $\{\zeta_k\}_{k\in \N}$ be a sequence of $(\rho_1,\rho_2)$-admissible point sets
in $D_{1/\e_k}$.
Then there exist a subsequence (not relabeled) and a Carath\'eodory function $W:D\times \M^{n\times d}\to \R^+$
such that for all open sets $A\in \calO(D)$ the functionals 
$\overline F_{\e_k}^A(\zeta_k)$ 
given by
(\ref{eq:def-F-var2})
%
%\begin{equation*}
%\overline F_{\e_k}^A(\zeta_k,u)\,:=\,
%F_{\e_k}^{A\cap  D_{1,2\e_k \rho_2}}(\zeta_k,u) ,
%\end{equation*}
%%
%where $F_{\e_k}^A(\zeta_k,\cdot)$ is defined as in \eqref{eq:def-ener-vol2},
satisfy
\begin{equation}
\Gamma-\lim_{k\to \infty} \overline F_{\e_k}^A(\zeta_k) = F^A
\end{equation}
where the integral
functional $F^A:L^{p}(D,\R^n)\to[0,+\infty]$ is defined by
\begin{equation}
F^A(u)=\begin{cases}\int_A W(x,\nabla
u(x))\ud x & \text{if } u\in W^{1,p}(A,\R^n), \cr +\infty
&\text{otherwise}.
\end{cases}
\end{equation}
In addition, $W$ is quasiconvex in its second variable and
satisfies a standard growth condition \eqref{eq:def-sgc} of order $p$.
\end{lemma}
%
% %
% \begin{remark}\label{rem:compact}
% We shall need a localized version of Lemma~\ref{lem:compact} in Subsection~\ref{sec:bded-set}, for which $\zeta$ is not defined on the whole space.
% For every integer $k\geq 1$ we let $\overline \zeta_k$ be a $(\rho_1,\rho_2)$-admissible point set in
% $D_k$, $\calT_k$ be a sub-tessellation of the Delaunay tessellation associated with $\overline \zeta_k$ such that
% \begin{itemize}
% \item the point set $\zeta_k$ defined as the set of vertices of $\calT_k$ is $(\rho_1,\rho_2)$-admissible in $D_k$,
% \item the edge lengths of $\calT_k$ are uniformly bounded by $2\rho_2$,
% \item $\{x\in D_k \,|\,d_\infty(x,\partial D_k )>2\rho_2 \} \subset \cup_{T\in \calT_k}\overline{T}$.
% \end{itemize}
% Then the conclusion of Lemma~\ref{lem:compact} holds for a
% subsequence of $\e_k=1/k$ if $F_{\e_k}^A(\zeta_k)$ denotes the functional defined by \eqref{eq:def-ener-vol2}, where 
% $\calT(\e_k \zeta_k)$ is replaced by $\e_k \calT_k$, and $\calN(\e_k \zeta_k)$ is replaced by the set of edges of $\e_k \calT_k$
% in \eqref{def:ener-nn-lr}, and \eqref{eq:def-volume}
% \end{remark}
% %
% %
%
\begin{proof}[Proof of Theorem \ref{th:rp-gen}]
%\Comment{AG} {\bf [The maths is as I intended]}
%The functional $\overline{F}_\e^A$ defined by \eqref{eq:def-F-var2} 
%satisfies the assumptions of
By  Lemma~\ref{lem:compact-local},
%Hence, almost surely, 
there exists almost surely an energy density $W:D\times \M^{n\times d} \to \R^+$ and a subsequence such that 
for all open sets $A\in \calO(D)$ the functionals $\overline F^A_\e(\xi^{1/\e})$
$\Gamma$-converge along the subsequence to 
$F^{A}:L^{p}(A,\R^n)\to[0,+\infty]$ defined by
\begin{equation*}
F^A(u)=\begin{cases}\int_A W(x,\nabla
u(x))\ud x & \text{if } u\in W^{1,p}(A,\R^n), \cr +\infty
&\text{otherwise}.
\end{cases}
\end{equation*}
As in \cite[Theorem~2]{Alicandro-Cicalese-Gloria-07b} we then appeal to 
the characterization of non-homogeneous quasiconvex functions by their minima: 
for almost every $x\in D$ and for all $\Lambda \in \M^{n\times d}$,
$$
W(x,\Lambda)\,=\,\lim_{\rho\to 0} \frac{1}{|Q_\rho|} 
%\inf
\inf_{v\in W^{1,p}_0(x+Q_\rho,\R^n)}
\bigg\{ \int_{x+Q_\rho}W(y,\Lambda+\nabla v(y))
%\,,\,
%v\in W^{1,p}_0(x+Q_\rho,\R^n)
\bigg\}.
$$
By Lemma~\ref{lem:BC}, for all $\rho>0$ small enough so that 
$x+\overline Q_{2 \rho} \subset D$, we have
\begin{multline*}
\inf_{v\in W^{1,p}_0(x+Q_\rho,\R^n)}
\bigg\{ \int_{x+Q_\rho}W(y,\Lambda+\nabla v(y))
%\,,\,v\in W^{1,p}_0(x+Q_\rho,\R^n)
\bigg\} \\
=\,\lim_{\e \to 0} \inf\bigg\{F_\e^{x+Q_\rho}(\xi^{1/\e},v)
\, \big| \,
v\in {\mathfrak{S}}_\e^{x+Q_\rho} (\xi^{1/\e}),
%v(y)=\varphi_\Lambda(y)
v \equiv \varphi_\Lambda
\mbox{ on } \e \xi^{1/\e} \cap (x+(Q_{\rho} \setminus Q_{\rho - 4 \e \rho_0}))
%\\
%\mbox{ for all }
%y\in \e  \xi^{1/\e} \mbox{ such that } d(y,\partial (x+Q_\rho))\leq 4\e \rho_0
\bigg\}.
\end{multline*}
On the one hand, as proved in Lemma~\ref{e2lemma}, for all $0<\alpha<1$ there exists an almost surely finite random variable
$R_0$ such that the measures $\xi$ and $\xi^{1/\e}$ coincide on $D_{R,R^\alpha}$ for all $R\geq R_0$. 
Since for all $\e$ small enough, $\{\e^{-1}y\,|\,y\in x+Q_\rho\}\subset  
D_{\e^{-1},{\e}^{-\alpha}}$, 
there exists $\e_0$ such that $\e\xi$ and $\e \xi^{1/\e}$ coincide on
$x+Q_\rho$ for all $\e<\e_0$. Therefore, for all $\e<\e_0$,
\begin{multline*}
\inf\bigg\{F_\e^{x+Q_\rho}(\xi^{1/\e},v)\,\big|\,v\in
{\mathfrak{S}}_\e^{x+Q_\rho} (\xi^{1/\e}),
v \equiv \varphi_\Lambda 
%v(y)=\varphi_\Lambda (y)
%\\
%\mbox{ for all }
\mbox{ on }
%y\in \e  \xi^{1/\e} \mbox{ such that } d(y,\partial (x+Q_\rho))\leq 4\e \rho_0
\e  \xi^{1/\e} \cap  (x+(Q_\rho \setminus Q_{\rho - 4 \e \rho_0}))
\bigg\}
\\
=\inf\bigg\{F_\e^{x+Q_\rho}(\xi,v)\,\big|
\,v\in {\mathfrak{S}}_\e^{x+Q_\rho} (\xi),
%v(y)=\varphi_\Lambda (y)
v \equiv \varphi_\Lambda 
%\\
%\mbox{ for all }
\mbox{ on }
%y\in \e  \xi \mbox{ such that } d(y,\partial (x+Q_\rho))\leq 4\e \rho_0\bigg\}.
\e  \xi \cap   (x+(Q_\rho \setminus Q_{\rho - 4\e \rho_0} )) \bigg\}.
\end{multline*}
On the other hand, as proved in Step~1 of the proof of \cite[Theorem~2]{Alicandro-Cicalese-Gloria-07b}, by a suitable application of the subadditive ergodic theorem,
\begin{multline*}
\lim_{\e \to 0} 
\inf\bigg\{F_\e^{x+Q_\rho}(\xi,v)\,\big| 
\,v\in {\mathfrak{S}}_\e^{x+Q_\rho} (\xi),
%v(y)=\varphi_\Lambda(y) 
v \equiv \varphi_\Lambda 
%\\
%\mbox{ for all }
\mbox{ on }
%y\in \e  \xi\mbox{ such that } d(y,\partial (x+Q_\rho))\leq 4\e \rho_0\bigg\}
\e  \xi \cap (x+(Q_\rho \setminus Q_{\rho - 4\e \rho_0} )) \bigg\}
\,=\,W_\ho(\Lambda)
\end{multline*}
almost surely.
Hence, $W(x,\cdot)$ coincides with $W_\ho$ for almost every $x\in D$.
In particular, by uniqueness of the limit, the entire sequence converges almost surely.
\end{proof}
%

%

%%%%%%%%%%%%%%%%%%%%%%%%%%%%%%%%%%%%%%%%%%%%%%%%%%%%%%%%%%%%%%%%%%%%%%%%%%
%%%%%%%%%%%%%%%%%%%%%%%%%%%%%%%%%%%%%%%%%%%%%%%%%%%%%%%%%%%%%%%%%%%%%%%%%%

\section*{Acknowledgements}

The authors acknowledge the support of INRIA through the ``Action de recherche collaborative'' DISCO.
This work was also supported by Ministry of Higher Education and Research,
Nord-Pas de Calais Regional Council and FEDER through the ``Contrat de
Projets Etat Region (CPER) 2007-2013''.

%%%%%%%%%%%%%%%%%%%%%%%%%%%%%%%%%%%%%%%%%%%%%%%%%%%%%%%%%%%%%%%%%%%%%%%%%%
%%%%%%%%%%%%%%%%%%%%%%%%%%%%%%%%%%%%%%%%%%%%%%%%%%%%%%%%%%%%%%%%%%%%%%%%%%

\bibliographystyle{amsalpha}

\end{document}